\documentclass[10pt]{amsart}
\usepackage[latin2]{inputenc}
\usepackage{amsmath}
\usepackage{graphicx}
\usepackage{amssymb}
\usepackage{esint}
\usepackage[dvipsnames]{xcolor}
\usepackage{tikz}
\usepackage{xxcolor}
\usepackage{floatrow}
\usepackage{color}
\usepackage{amsthm}
\usepackage{epsfig}
\usepackage[english]{babel}
\usepackage{hyperref}
\usepackage[normalem]{ulem}
\usepackage{pgfplots}
\pgfplotsset{compat=1.14}

\hypersetup{
    colorlinks,
    linkcolor={red!80!black},
    citecolor={blue!50!black},
    urlcolor={blue!80!black}
}

\allowdisplaybreaks[2]

\newcommand{\N}{\mathbb{N}}

\newcommand{\R}{\mathbb{R}}

\newcommand{\Rd}{{\mathbb{R}^d}}
\def\Xint#1{\mathchoice
	{\XXint\displaystyle\textstyle{#1}}%
	{\XXint\textstyle\scriptstyle{#1}}%
	{\XXint\scriptstyle\scriptscriptstyle{#1}}%
	{\XXint\scriptscriptstyle\scriptscriptstyle{#1}}%
	\!\int}
\def\XXint#1#2#3{{\setbox0=\hbox{$#1{#2#3}{\int}$ }
		\vcenter{\hbox{$#2#3$ }}\kern-.6\wd0}}

\def\averageint{\Xint-}
\def\dashint{\Xint-}

\DeclareMathOperator*{\supp}{\rm supp}

\newtheorem{theorem}{Theorem}[section]
\newtheorem{lemma}[theorem]{Lemma}

\newtheorem{definition}[theorem]{Definition}

\theoremstyle{remark}
\newtheorem{remark}[theorem]{Remark}

\newcommand\restr[2]{{
		\left.\kern-\nulldelimiterspace 
		#1 
		\vphantom{\big|} 
		\right|_{#2} 
}}

\newcommand{\pt}{\partial_t}

\renewcommand{\d}{{\rm d}}
\newcommand{\dt}{{\,\mathrm{d}t}}
\newcommand{\dx}{{\,\mathrm{d}x}}
\newcommand{\ds}{{\,\mathrm{d}s}}

\newcommand{\eps}{{\varepsilon}}

\renewcommand{\i}{\ifmmode\mathit{\mathchar"7010 }\else\char"10 \fi}
\renewcommand{\j}{\ifmmode\mathit{\mathchar"7011 }\else\char"11 \fi}

\newcommand{\dist}{\hbox{dist}}

\renewcommand{\le}{\leq} 
\renewcommand{\ge}{\geq}

\def\char{{1\!\mbox{\rm l}}}

\usepackage[colorinlistoftodos]{todonotes}

\begin{document}
\title[Sharp criteria for free boundary propagation in thin liquid films]{Sharp criteria for the waiting time phenomenon in solutions to the thin-film equation}

\subjclass[2010]{35K25, 35K55, 35K65, 35Q35, 35R35, 76D08}
\keywords{thin-film equation, higher-order degenerate parabolic equation, free boundary problem, finite speed of propagation, waiting time phenomenon}

\author{Nicola De Nitti}
\address[N. De Nitti]{Friedrich-Alexander-Universit\"at Erlangen-N\"urnberg, Department of Mathematics, Chair in Applied Analysis -- Alexander von Humboldt Professorship, Cauerstr. 11, 91058 Erlangen, Germany.}
\email{nicola.de.nitti@fau.de}

\author{Julian Fischer}
\address[J. Fischer]{Institute of Science and Technology Austria (IST Austria), Am Campus 1, 3400 Klosterneuburg, Austria.}
\email{julian.fischer@ist.ac.at}

\begin{abstract}
We establish sharp criteria for the instantaneous propagation of free boundaries in solutions to the thin-film equation. The criteria are formulated in terms of the initial distribution of mass (as opposed to previous almost-optimal results), reflecting the fact that mass is a locally conserved quantity for the thin-film equation. In the regime of weak slippage, our criteria are at the same time necessary and sufficient.
The proof of our upper bounds on free boundary propagation is based on a strategy of ``propagation of degeneracy'' down to arbitrarily small spatial scales: We combine estimates on the local mass and estimates on energies to show that ``degeneracy'' on a certain space-time cylinder entails ``degeneracy'' on a spatially smaller space-time cylinder with the same time horizon. The derivation of our lower bounds on free boundary propagation is based on a combination of a monotone quantity and almost optimal estimates established previously by the second author with a new estimate connecting motion of mass to entropy production.
\end{abstract}

\maketitle
\section{Introduction}
\label{SectionIntro}

\subsection{The thin-film equation}
\label{SubSectionThinFilm}

The thin-film equation
\begin{align}
\label{ThinFilmEquation}
\partial_t u = -\nabla \cdot (u^n \nabla \Delta u)
\end{align}
(with the positive real parameter $n>0$) describes the surface-tension-driven evolution of the height $u(x,t)$ of a viscous thin liquid film on a flat surface.
Like its second-order sibling, the porous medium equation
\begin{align*}
\partial_t u = \Delta u^m = m \nabla \cdot (u^{m-1}\nabla u)
\end{align*}
(with $m>1$; see e.\,g.\ \cite{VazquezPM} for an overview of the corresponding theory), the thin-film equation gives rise to a free boundary problem, the free boundary being the boundary of the liquid film $\partial \{u(\cdot,t)>0\}$. The dynamics of the thin-film equation \eqref{ThinFilmEquation} is mostly of interest in the regime $n\in (0,3)$, as for $n\geq 3$ it is conjectured that the support of solutions remains constant in time. Physically, the parameter $n$ is determined by the boundary condition for the flow at the liquid-solid interface: The case $n = 3$ corresponds to a no-slip boundary condition \cite{RevModPhys}; $n = 2$ takes into account -- roughly speaking -- the Navier slip condition (see \cite{Greenspan, JaegerMikelic}), and various parameters $n \in (1, 3)$ have been suggested to model the effects of stronger ($1 < n < 2$) or weaker ($2 < n < 3$) slippage \cite{Greenspan}. The case $n=1$ arises in the lubrication approximation of the Darcy's flow in the Hele-Shaw cell \cite{GiacomelliOttoHS}.

In the present work, we are interested in the qualitative behavior of the free boundary $\partial \{u(\cdot,t)>0\}$ in the so-called case of complete wetting. Depending on the growth of the initial data $u_0$ near the free boundary, a \emph{waiting time phenomenon} may occur: If the initial data $u_0$ are ``flat enough'' near some point $x_0$ on the initial free boundary -- namely, if $u_0$ grows at most like $|x-x_0|^{4/n}$ near $x_0$ -- , the free boundary will locally remain stationary (or at most move backward) for some time before it finally starts moving forward (see \cite{DalPassoGiacomelliGruen,GruenWTWS,GiacomelliGruen}). The amount of time that passes before the free boundary moves beyond its initial location is called the \emph{waiting time}. On the other hand, in the regime of weak slippage $n \in (2,3)$, it is known that the free boundary will start moving forward instantaneously if the initial data $u_0$ grow steeper than $|x-x_0|^{4/n}$ near the initial free boundary \cite{FischerARMA,FischerAHP}; in the case $n =2$, a similar result holds up to a logarithmic correction term. The restriction $n\geq 2$ in the results of \cite{FischerARMA,FischerAHP} is optimal, as for $n<2$ the stationary state $u(x,t)=(x-x_0)_+^2$ would provide a counterexample. However, even in the regime $n \in (2,3)$ there is a small gap between the sufficient conditions for a waiting time in \cite{DalPassoGiacomelliGruen,GruenWTWS, GiacomelliGruen} and the sufficient conditions for instantaneous forward motion in \cite{FischerARMA,FischerAHP}: This gap is not in terms of the critical growth exponent $4/n$ (which is inferred from the scaling of the equation, see \cite[Section 7]{DalPassoGiacomelliGruen}), but in terms of the norms used to formulate the growth condition. It is the goal of the present work to close this gap, providing a condition for the occurrence of a waiting time phenomenon for a higher-order degenerate parabolic equation which is at the same time necessary and sufficient. Even though the remaining gap is only in terms of norms and not in terms of scaling, closing it requires substantial additional ideas; see Section~\ref{SectionSummary} below for a comparison of our new results to the previous ones in the literature, Section~\ref{SectionMainResults} for precise statements of our theorems, and Section~\ref{SectionStrategy} for a summary of the strategies employed to carry out the proofs.

In contrast to the porous medium equation, due to its fourth-order structure the thin-film equation does not give rise to a comparison principle. In the parameter range $n<\smash{\frac{3}{2}}$, the support of solutions to the thin-film equation may even shrink as shown for example by the moving front solution $u(x,t)=\smash{(x-c_n t)_+^{3/n}}$. Furthermore, many techniques for second-order equations -- in particular from regularity theory -- are not applicable to the thin-film equation. For these reasons, the analysis of the qualitative behavior of the thin-film equation -- and in particular the derivation of lower bounds on free boundary propagation, first accomplished in \cite{FischerARMA,FischerAHP} -- are substantially more challenging than in the case of the porous medium equation.

Due to the fourth order structure of the thin-film equation -- and unlike in the case of the second-order porous medium equation -- , it is also necessary to prescribe an additional boundary condition at the free boundary $\partial \{u(\cdot,t)>0\}$ (in addition to the natural boundary condition $u=0$) in order to prevent ill-posedness \cite{BerettaBertschDalPasso}. Energetic considerations suggest to prescribe the contact angle -- that is, the slope $|\nabla u|$ -- at the free boundary according to Young's law. The case of zero contact angle $|\nabla u|=0$ is called the case of ``complete wetting'', while the case of a fixed positive contact angle $|\nabla u|=b>0$ is known as the case of ``partial wetting''.

In the last decades, an extensive theory of weak solution concepts (see \cite{BerettaBertschDalPasso,BernisFinite2,BernisFriedman,ThinViscous,
WeakInitialTrace,OnAFourthOrder,CauchyGruen}) and strong solution concepts (see \cite{GiacomelliGnannKnuepferOtto,GiacomelliKnuepfer,GiacomelliKnuepferOtto,
GnannIbrahimMasmoudi,GnannPetrache,Gnann,Gnann2,John}) has been developed for the case of vanishing contact angle $|\nabla u|=0$ on $\partial \supp u(\cdot,t)$. However, to date no uniqueness result is known for weak solution concepts in the presence of a free boundary (except for Dirac initial data in the case $n=1$; see \cite{MajdoubMasmoudiTayachi}), while the strong solution concepts are so far limited to local-in-time existence results or small perturbations of self-similar solutions or steady-states. Nevertheless, there is a rich theory of qualitative behavior of solutions to the thin-film equation. The long-time behavior of the thin-film equation has been studied e.\,g.\ in \cite{CarrilloToscani,McCannSeis,Seis}. Finite speed of propagation results for the free boundary have been proven in \cite{BernisFinite,BernisFinite2,ThinViscous,GruenOptimalRatePropagation,
HulshofShishkov}. Sufficient conditions for waiting times have been established rigorously in \cite{DalPassoGiacomelliGruen,GiacomelliGruen}. A formal analysis of the waiting time behavior has been performed in \cite{King}. Based on the discovery of certain new monotonicity formulas, lower bounds on free boundary propagation have been proven in \cite{FischerARMA,FischerJDE,FischerAHP}. For more complex (S)PDEs of thin-film type, see for example \cite{AnsiniGiacomelli,BertozziPughBlowup,FischerGruen,GessGnann,GiacomelliShishkov}, though this list is far from exhaustive.

In the case of partial wetting $|\nabla u|=b>0$ on $\partial \supp u(\cdot,t)$ for some constant $b>0$, the mathematical theory for the thin-film equation is more limited and consists mostly of some existence (and, for strong solution concepts, uniqueness) results; see \cite{BertschGiacomelliKarali,Mellet,Otto} for weak solution concepts and \cite{Degtyarev,Knuepfer,Knuepfer2} for strong solution concepts.

Despite the lack of a comparison principle, the thin-film equation is one of the two notable examples of a nonnegativity-preserving fourth-order equation, the other one being the Derrida-Lebowitz-Speer-Spohn equation (DLSS equation) (see e.\,g.\ \cite{DerridaLebowitzSpeerSpohn,DerivationQDD,FischerDLSSUniqueness,GiaSavTosc,
JuengelMatthesReview,JuengelMatthesExistence,JuengelPinnau}). Recall that the standard linear parabolic equation $\partial_t u = -\Delta^2 u$ fails to preserve positivity.
In contrast to the thin-film equation, solutions to the DLSS equation feature infinite speed of propagation \cite{FischerInfiniteSpeedDLSS}.
For further classes of nonnegativity-preserving higher-order parabolic equations, see e.\,g.\  \cite{BukalJuengelMatthes,LisiniMatthesSavare,MatthesMcCannSavare}.

\subsection{Informal summary of the results}
\label{SectionSummary}

In the present work, in the parameter regime $2<n<3$ we provide conditions on the initial data $u_0$ which are both necessary and sufficient for instantaneous forward motion of the free boundary in solutions to the thin-film equation \eqref{ThinFilmEquation} in the case of zero contact angle $|\nabla u|=0$ on $\partial \supp u(\cdot,t)$.

To give one example of our results, consider the one-dimensional thin-film equation $\partial_t u = -(u^n u_{xxx})_x$ with compactly supported nonnegative initial data $u_0\in H^1(\mathbb{R})$. Denote by $x_0$ the leftmost point in the support of $u_0$. In the regime $2<n<3$, we prove that instantaneous forward motion of the free boundary at $x_0$ occurs if and only if $u_0$ grows faster than $\smash{(x-x_0)_+^{4/n}}$ near the free boundary $x_0$ in the sense of ``averages of the mass''
\begin{align}
\label{ConditionFreeBoundaryPropagation1d}
\limsup_{r\rightarrow 0} r^{-4/n} \dashint_{(x_0,x_0+r)} u_0 \,\d x=\infty.
\end{align}
In other words, a waiting time phenomenon occurs if and only if the opposite condition
\begin{align}
\label{ConditionFreeBoundaryPropagation1dFinite}
\limsup_{r\rightarrow 0} r^{-4/n} \dashint_{(x_0,x_0+r)} u_0 \,\d x<\infty
\end{align}
holds true.

Our new results differ from the previous results in the literature as follows:
\begin{itemize}
\item The best previously known sufficient condition for the occurrence of a waiting time phenomenon for the thin-film equation for $n\in [2,3)$ was
\begin{align}
\label{ConditionWaitingTimeGiacomelliGruen}
\limsup_{r\rightarrow 0} r^{-4/n+1} \bigg(\dashint_{(x_0,x_0+r)} |\nabla u_0|^2 \,\d x\bigg)^{1/2}<\infty
\end{align}
as derived by Dal~Passo, Giacomelli, and Gr\"un in  \cite{DalPassoGiacomelliGruen}.
While for ``regular'' initial data like $u_0(x)=\smash{(x-x_0)_+^\beta}$ near $x_0$ for some $\beta>0$ the condition \eqref{ConditionWaitingTimeGiacomelliGruen} is equivalent to our condition \eqref{ConditionFreeBoundaryPropagation1d}, it fails to capture cases of ``irregular'' initial data: For example, the oscillatory initial data
\begin{align}
\label{OscillatoryInitialData}
u_0(x):=\bigg(2+\sin \frac{1}{x-x_0}\bigg)(x-x_0)_+^{4/n}
\end{align}
meets our new sufficient criterion for the occurrence of a waiting time \eqref{ConditionFreeBoundaryPropagation1dFinite} but fails to meet the previously known condition \eqref{ConditionWaitingTimeGiacomelliGruen}. For a plot of the example \eqref{OscillatoryInitialData}, see Figure~\ref{FigureOscillatory}.
\item The only previous results guaranteeing instantaneous forward motion of the free boundary in solutions to the thin-film equation -- as derived in a series of papers by the second author \cite{FischerARMA,FischerAHP} -- required the slightly stronger condition
\begin{align}
\label{ConditionForwardMotionFischer}
\limsup_{r\rightarrow 0} r^{-4/n} \bigg(\dashint_{(x_0,x_0+r)} u_0^p \,\d x\bigg)^{1/p}=\infty
\end{align}
for a certain $p=p(n)\in (0,1)$, with typically $0<p(n)\leq \frac{1}{2}$. While again for ``regular'' initial data like $u_0(x)=\smash{(x-x_0)_+^\beta}$ near $x_0$ the condition \eqref{ConditionForwardMotionFischer} is equivalent to our new condition \eqref{ConditionFreeBoundaryPropagation1dFinite}, the two conditions differ in the case of ``concentrated'' initial data: For example, letting $\varphi:\mathbb{R}\rightarrow \mathbb{R}^+_0$ be a bump function supported in $[0,1]$, for the initial data
\begin{align}
\label{ConcentratedInitialData}
u_0(x)
:=(x-x_0)_+^{4/n}
+(x-x_0)_+^{4/n-\delta}\cdot \sum_{k=2}^\infty k^{2} \varphi\bigg(k^2 \bigg(x-x_0-\frac{1}{k}\bigg)\bigg)
\end{align}
(for $\delta>0$ small enough) our new condition \eqref{ConditionFreeBoundaryPropagation1d} for instantaneous forward motion of the free boundary is satisfied, but the previously known condition \eqref{ConditionForwardMotionFischer} is not. See Figure~\ref{FigureConcentrated} for an illustration of the example \eqref{ConcentratedInitialData}.
\item We also obtain optimal upper and lower bounds for waiting times, which are both formulated in terms of the quantity
\begin{align}
\label{QuantityWaitingTime}
\sup_{r>0} r^{-4/n} \dashint_{(x_0,x_0+r)} u_0 \,\d x
\end{align}
and differ from each other only by a constant factor.
\end{itemize}

\begin{figure}
\begin{tikzpicture}
\draw (6.5,0.473) -- (-2.0,0.473);
\begin{axis}[hide axis,domain=0.001:0.9,legend pos=outer north east,every axis plot post/.append style={mark=none,samples=300,smooth}]
\addplot {x*x*(2+sin(200.0/x))};
\end{axis}
\end{tikzpicture}
\caption{Plot of the example \eqref{OscillatoryInitialData}. While the initial data $u_0$ are clearly bounded from above and from below by a multiple of $(x-x_0)^{4/n}$, due to the rapid oscillations near the free boundary the limit \eqref{ConditionWaitingTimeGiacomelliGruen} is infinite. As a result, the sufficient criterion for waiting times from \cite{DalPassoGiacomelliGruen} is not applicable. In contrast, our sufficient condition in Theorem~\ref{th:masssuf} shows that for this initial data indeed a waiting time phenomenon occurs.
\label{FigureOscillatory}
}
\end{figure}
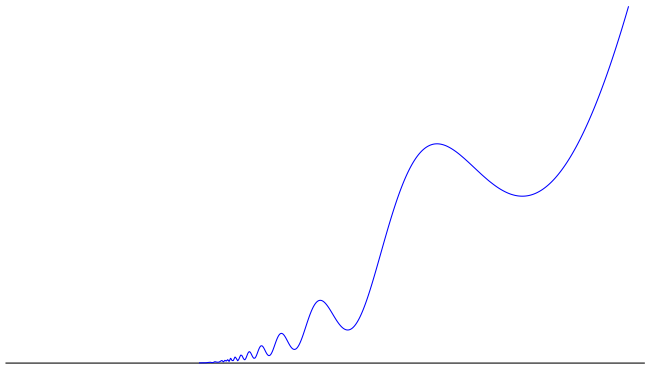

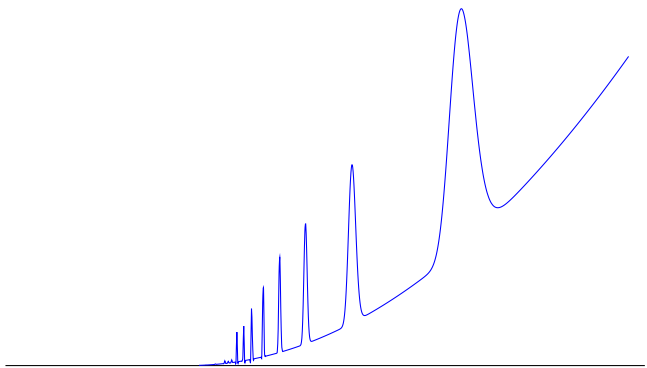
\begin{figure}
\begin{tikzpicture}
\draw (6.5,0.473) -- (-2.0,0.473);
\begin{axis}[hide axis,domain=0.02:0.9,legend pos=outer north east,every axis plot post/.append style={mark=none,samples=500,smooth}]
\addplot {x*x+(x^0.8)*((sin(150.0/x)*sin(150.0/x))^(5.0+4.0/x))};
\end{axis}
\end{tikzpicture}
\caption{Illustration of the example \eqref{ConcentratedInitialData}. The initial data features infinitely many ``bumps'' accumulating at $x_0$. The ``bumps'' near a point $x>x_0$ have mass of order $(x-x_0)^{4/n-\delta}$ but width of order $|x-x_0|^2$. As a consequence of the mass estimate for the bumps, our sufficient condition for instantaneous forward motion of the free boundary in Theorem~\ref{th:massnec} is applicable. In contrast, the sufficient conditions for instantaneous forward motion from \cite{FischerARMA,FischerAHP} are not applicable for $\delta>0$ small enough, as the increasingly strong concentration of the bumps cause the limit in \eqref{ConditionForwardMotionFischer} to be finite.
\label{FigureConcentrated}
}
\end{figure}

Our sufficient criterion for a waiting time \eqref{ConditionFreeBoundaryPropagation1d} is not limited to the regime $n\in (2,3)$, but holds for the full range $n \in (1,3)$. However, the stationary state $u(x,t)=(x-x_0)_+^2$ shows that in the regime $n<2$ one cannot expect a condition like \eqref{ConditionFreeBoundaryPropagation1dFinite} to be sufficient for instantaneous forward motion of the free boundary, as $(x-x_0)_+^2$ grows steeper than $\smash{(x-x_0)_+^{4/n}}$ in this regime. Nevertheless, the constructions in \cite{FischerAHP} show that our condition \eqref{ConditionFreeBoundaryPropagation1d} is in fact sharp among all conditions formulated in terms of the growth of the initial data at the free boundary: In \cite[Theorem 3]{FischerAHP} it is shown that there exist initial data with only slightly steeper growth than $\smash{(x-x_0)_+^{4/n}}$ for which instantaneous forward motion occurs.

{\bf Notation.} Throughout the paper, we use standard notation for Lebesgue and Sobolev spaces. For a domain $\Omega\subset \mathbb{R}^d$ we denote for $p\geq 1$ by $L^p(\Omega)$ the space of all measurable functions $f$ with finite norm $||f||_{L^p(\Omega)}:=(\int_\Omega |f|^p \dx)^{1/p}$.
The Sobolev space $W^{1,p}(\Omega)$, $1\leq p\leq \infty$, consists of all measurable functions $f\in L^p(\Omega)$ whose distributional derivative $\nabla f$ belongs to $L^p(\Omega)$; it is equipped with the norm $||f||_{W^{1,p}}:=(\int_\Omega |f|^p + |\nabla f|^p \dx)^{1/p}$. Similarly, we define higher-order Sobolev spaces $W^{k,p}(\Omega)$, $k\geq 2$, consisting of those functions in $W^{k-1,p}(\Omega)$ whose $k$th distributional derivatives belongs to $L^p(\Omega)$.
We also use the standard abbreviation $H^k(\Omega):=W^{k,2}(\Omega)$.
For a function $f:\Omega\times [0,T]\rightarrow \mathbb{R}$ depending on space and time, we denote by $\nabla$ and $\Delta$ the (weak) gradient and the (weak) Laplacian with respect to spatial coordinates only. The (weak) time derivative of $f$ is denoted by $\partial_t f$.
As usual, for a Banach space $X$ we denote by $X'$ its dual.
Given a Banach space $X$, by $L^p([0,T];X)$ we denote the usual Lebesgue-Bochner space of strongly measurable maps $f:[0,T]\rightarrow X$ with $||f||_{L^p([0,T];X)}^p:=\int_{[0,T]} |f|_X^p \dt<\infty$.
By $B_r(x)$ we denote the ball of radius $r$ around the point $x$.

\section{Main Results}
\label{SectionMainResults}

The rigorous definition of a \emph{waiting time} which our results refer to is given as follows.
\begin{definition}
	\label{DefinitionWaitingTime}
	Let $u_0\in L^1(\mathbb{R}^d)$ and $u\in L^\infty([0,T);L^1(\mathbb{R}^d))$. For any point $x_0\in \mathbb{R}^d\setminus \supp u_0$ in the complement of the support of $u_0$, we define the \emph{waiting time} $T^*$ of $u$ at $x_0$ as
	\begin{align*}
	T^* := \operatorname{essinf} \{t>0:x_0\in \supp u(\cdot,t)\},
	\end{align*}
	where $\supp u(\cdot,t)$ is understood in the sense of support of a distribution.
	
	For any point $x_0\in \partial \supp u_0$ on the boundary of the initial support, we define the waiting time $T^*$ of $u$ at $x_0$ as
	\begin{align*}
	T^* := \operatorname{essinf} \{t>0:x_0\notin \overline{\mathbb{R}^d \setminus \supp u(\cdot,t)}\}.
	\end{align*}
\end{definition}
In other words, for a point $x_0$ which lies outside of the support of the initial data, we define the waiting time $T^*$ to be the first time at which the support of the solution $u$ reaches $x_0$. For a point $x_0$ on the initial free boundary $\partial \supp u_0$, we define the waiting time to be the first time at which $x_0$ is contained in the interior of the support of the solution $u$.

We defer the (rather technical) definitions of solutions to the thin-film equation  and first state our main results (Theorem~\ref{th:masssuf} and Theorem~\ref{th:massnec}).
In the regime $n \in (1,3)$, we provide the following sufficient condition for the occurrence of a waiting time phenomenon, along with lower bounds on the waiting time.
\begin{theorem}\label{th:masssuf}
Let $d \in \{1,2,3\}$ and $n \in (1,3)$. Let $u_0\in H^1(\Rd)$ be compactly supported and nonnegative. In the case $n\in [2,3)$, let $u:\R^d \times [0,T) \to \R$ be an energy-dissipating weak solution to the thin-film equation \eqref{ThinFilmEquation} with zero contact angle and initial data $u_0$ in the sense of Definition~\ref{DefinitionEnergyDissipatingWeakSolution}. In the case $n\in (1,2)$, let $u:\R^d \times [0,T) \to \R$ instead be a weak solution to the thin-film equation \eqref{ThinFilmEquation} with zero contact angle and initial data $u_0$ in the sense of Definition~\ref{DefinitionWeakSolution}, and assume that $u$ has been constructed by the approximation procedure in \cite{ThinViscous}.

Let $x_0 \in \partial \supp u_0 ~\smash{\bigcup}~ (\Rd \setminus \supp u_0)$ be a point on the boundary or outside of the support of the initial data. Suppose that there exists a constant $\kappa>0$ such that for all $r>0$ the estimate
\begin{align}
\label{ConditionWaitingTime}
\averageint_{B_r(x_0)} u_0 \dx \le \kappa\, r^{\frac{4}{n}}
\end{align}
holds.
If $x_0\in \partial \supp u_0$, suppose furthermore that $\supp u_0$ satisfies an exterior cone condition at $x_0$ with some positive opening angle $\lambda>0$.\footnote{As usual, we say that a closed set $U\subset \mathbb{R}^d$ satisfies an exterior cone condition at $x_0\in \partial U$ with opening angle $\lambda>0$ if its complement $\Rd\setminus U$ contains a cone $C_{x_0}$ with tip $x_0$, opening angle $\lambda>0$, and arbitrary axis and height. In the one-dimensional case, the notion of ``exterior cone condition'' reduces to the requirement that either $(x_0,x_0+\delta) \cap \supp u_0$ or $(x_0-\delta,x_0) \cap \supp u_0$ is empty for some $\delta>0$ small enough, and the notion of ``opening angle'' becomes irrelevant.}

Then $u$ has a positive waiting time $T^*$ at $x_0$ (in the sense of Definition~\ref{DefinitionWaitingTime}) and there exists a constant $c$ (depending only on $d$, $n$, and possibly $\lambda$) such that the waiting time $T^*$ is bounded from below by
\begin{align*}
T^* \ge c \, \kappa^{-n}.
\end{align*}
\end{theorem}
In the regime of strong slippage $n\in(1,2)$, our preceding sufficient condition for a waiting time phenomenon is not a necessary condition, as the counterexample $u(x,t)=(x-x_0)_+^2$ demonstrates. Nevertheless, the approach of \cite{FischerAHP} shows that our sufficient condition for a waiting time phenomenon is (at least in one dimension $d=1$) optimal among all conditions formulated in terms of the growth of the initial data near the free boundary.

On the other hand, in the regime $n\in(2,3)$ our preceding sufficient condition for the occurrence of a waiting time phenomenon is also a necessary condition, as our next result shows. Furthermore, the lower bounds on the waiting time in Theorem~\ref{th:masssuf} above are optimal up to a universal constant factor.
\begin{theorem}\label{th:massnec}
Let $d=1$ and let $n\in(2,3)$. Let $u_0\in H^1(\mathbb{R})$ be compactly supported and nonnegative.  Let $u:\R^d \times [0,T) \to \R$ be an energy-dissipating weak solution to the thin-film equation \eqref{ThinFilmEquation} with zero contact angle and initial data $u_0$ in the sense of Definition~\ref{DefinitionEnergyDissipatingWeakSolution}. Let $x_0\in \partial \supp u_0~\smash{\bigcup}~(\mathbb{R}^d\setminus \supp u_0)$ be a point on the boundary or outside of the support of the initial data.
Then there exists a constant $C$ (depending only on $n$ and $d$) such that the
waiting time $T^*$ of $u$ at $x_0$ (in the sense of Definition~\ref{DefinitionWaitingTime}) is bounded from above by 
\begin{equation}
T^* \le C\left(\sup_{r>0} r^{-\frac{4}{n}} \dashint_{(x_0-r,x_0+r)} u_0 \dx \right)^{- n}.
\end{equation}
In particular, if the initial data $u_0$ satisfy
\begin{align*}
\limsup_{r\rightarrow 0} r^{-\frac{4}{n}} \dashint_{(x_0-r,x_0+r)} u_0 \dx = \infty
\end{align*}
at a point on the initial free boundary $x_0\in \partial \supp u_0$, the free boundary starts moving forward immediately at $x_0$, without waiting time.
\end{theorem}

\begin{remark}
In the multidimensional case $d\in \{2,3\}$, by combining the ideas of our proof of Theorem~\ref{th:massnec} with the approach used for the multidimensional case in \cite{FischerARMA}, one could prove a similar upper bound on the waiting time for $n\in (2,3)$, namely a bound of the form
\begin{align*}
T^* \leq C(d,n)
\bigg(\limsup_{\delta\rightarrow 0} \limsup_{r\rightarrow 0} r^{-4/n} \dashint_{(\partial \supp u_0 \cap B_\delta(x_0))+B_r} u_0 \,\d x\bigg)^{-n}
\end{align*}
for any point $x_0\in \partial \supp u_0$ near which $\partial \supp u_0$ is a $C^4$ manifold.
However, due to the already substantial length of the present paper we refrain from carrying out the estimates.
\end{remark}

Let us now state the precise definitions of solutions to the thin-film equation that our main results are concerned with.
For $d \in \{2,3\}$ and the parameter range
\begin{align*}
n \in \left(2-\sqrt{\frac{8}{8+d}}, 3\right),
\end{align*}
in \cite{CauchyGruen} an existence result has been proven for the following class of solutions to the thin-film equation. Earlier results of \cite{BernisFinite2} show the same existence result in $d=1$ for $n\in \left(\frac{1}{2},3\right)$.
\begin{definition}[Energy-dissipating weak solutions]
	\label{DefinitionEnergyDissipatingWeakSolution}
	Let $d\in \{1,2,3\}$ and $n \in (2,3)$. Let $T>0$ and let $u_0\in \smash{H^1(\Rd)}$ have compact support.
	We call a nonnegative function $\smash{u\in L^\infty([0,T); H^1(\Rd)\cap L^1(\Rd))}$, $u\geq 0$, an \emph{energy-dissipating weak solution of the thin-film equation with zero contact angle and initial data $u_0$} if the following conditions are satisfied:
	\begin{itemize}
		\item[a)] We have $\nabla u^{\frac{n+2}{6}} \in L^6(\Rd \times [0,T))$, $u^{\frac{n-2}{2}}\nabla u \otimes D^2 u \in L^2(\Rd \times [0,T))$, and $\chi_{\{u > 0\}}u^{\frac{n}{2}}\nabla\Delta u \in L^2(\Rd \times [0,T))$.
		\item[b)] For all $\alpha \in (\max\left\{-1,\frac{1}{2} - n\right\}, 2-n) \setminus \{0\}$, we have $D^2 u^{\frac{1+n+\alpha}{2}} \in L^2(\Rd \times [0,T))$ and $\nabla u^{\frac{1+n+\alpha}{4}} \in L^4(\Rd \times [0,T))$.
		\item[c)] It holds that $u \in H^1_{loc}([0,T);(W^{1,p}(\Rd))')$ for all $p > \frac{4d}{2d + n(2-d)}$.
		\item[d)] For any $\psi \in L^2([0,T),W^{1,\infty}(\Rd))$ and any $T>0$, we have
		\begin{align*}
		\int_{0}^{T} \langle \pt u, \psi \rangle_{(W^{1,p}(\Rd))'\times W^{1,p}(\Rd)} \dt = \int_0^T \int_{\Rd \cap \{u>0 \}} u^n \nabla \Delta u \cdot \nabla \psi \dx \dt.
		\end{align*}
		\item[e)] $u$ attains its initial data $u_0$ in the sense $\lim\limits_{t\to 0}u(\cdot,t)=u_0(\cdot)$ in $L^1(\Rd)$.
	\end{itemize}
\end{definition}

In the parameter range $n\in (1,2)$, we need to resort to a different solution concept, at least in case $d\in \{2,3\}$, as in this case the existence of energy-dissipating weak solutions is unknown.
\begin{definition}[Weak solutions]
	\label{DefinitionWeakSolution}
	Let $d\in \{1,2,3\}$ and $n \in (\frac{1}{8},2)$. Let $T>0$ and let $u_0\in \smash{H^1(\Rd)}$ have compact support.
	We say that a nonnegative function $u\in L^\infty([0,T); H^1(\Rd)\cap L^1(\Rd))$ is a \emph{weak solution of \eqref{ThinFilmEquation} with zero contact angle and initial data $u_0$} if the following conditions are satisfied:
	\begin{enumerate}
		\item[a)] $u \in H^1_{loc}\left([0,T);(W^{1,p}(\Rd))' \right)$ for all $p > \frac{4d}{2d+n(2-d)}$;
		\item[b)] For any $\alpha \in (\max\left\{-1,\frac{1}{2} - n\right\}, 2-n) \setminus \{ 0\}$, we have $D^2 u^{\frac{1+n+\alpha}{2}} \in L^2(\Rd \times [0,T))$ and $\nabla u^{\frac{1+n+\alpha}{4}} \in L^4(\Rd \times [0,T))$.
		\item[c)] for any $\psi \in L^\infty([0,T);C^3_{c}(\Rd))$
		we have for any $T>0$
		\begin{align*}
		&\int_0^T \langle \pt u,\psi \rangle_{(W^{1,p}(\Omega))'\times W^{1,p}(\Omega)} \dt 
		\\
		&= \int_0^T\int_{\Rd \cap  \{u>0\}} u^n \nabla u \cdot \nabla \Delta \psi \dx \dt
		\\&\quad + n \int_0^T\int_{\Rd \cap  \{u>0\}} u^{n-1}\nabla u \cdot D^2\psi \cdot \nabla u \dx \dt
		\\&\quad
		+ \frac{n}{2} \int_0^T\int_{\Rd \cap \{u>0\}} u^{n-1}|\nabla u|^2\Delta \psi \dx \dt
		\\&\quad 
		+ \frac{n(n-1)}{2} \int_0^T\int_{\Rd \cap  \{u>0\}} u^{n-2}|\nabla u|^2\nabla u \cdot \nabla \psi \dx \dt.
		\end{align*}
		\item[d)] $u$ attains its initial data $u_0$ in the sense $\lim\limits_{t\to 0}u(\cdot,t)=u_0(\cdot)$  in $L^1(\Rd)$.
	\end{enumerate}
\end{definition}

\section{Strategies for the proofs of the main results}
\label{SectionStrategy}

\subsection{Strategy for the lower bounds on free boundary propagation}
\label{SubSectionStrategyLB}

Our argument for the lower bounds on free boundary propagation for the thin-film equation relies in parts crucially on the results and strategies of the previous works by the second author \cite{FischerARMA,FischerAHP}.
In the particular case of one dimension $d=1$, the key results of \cite{FischerARMA,FischerAHP} may be summarized as follows: For $n\in (2,3)$, for any point $x_0$ on the boundary or outside of the support of the initial data $u_0$ the waiting time is bounded from above by
\begin{align*}
T^* \leq C \Bigg(\sup_{r>0} r^{-4/n} \bigg(\dashint_{(x_0-r,x_0+r)} u_0^p \,\d x\bigg)^{1/p}\Bigg)^n,
\end{align*}
where $C>0$ and $p\in (0,1)$ depend only on $n$. The results of \cite{FischerARMA,FischerAHP} are based on the discovery of certain new monotonicity formulas for solutions to the thin-film equation, taking the form of a weighted entropy inequality
\begin{align}
\label{MonotonicityFormula}
\frac{\mathrm{d}}{\mathrm{d}t} \int_{\mathbb{R}} u^{1+\alpha} |x-x_0|^{\gamma} \,\d x
\geq c \int_{\mathbb{R}} u^{1+\alpha+n} |x-x_0|^{\gamma-4} + |\nabla u^{\frac{1+\alpha+n}{4}}|^4 |x-x_0|^{\gamma} \,\d x
\end{align}
and being valid for suitable $-1<\alpha<0$ and suitable $\gamma<-1$, as long as the support of the solution $u(\cdot,t)$ does not touch the singularity of the weight at $x_0$ {(for a rigorous statement of this inequality, see Theorem~\ref{MonotonicityFormulaRigorous} in the appendix)}. The monotonicity formula enables one to apply a differential inequality argument due to Chipot and Sideris \cite{ChipotSideris}: Suppose, for the sake of simplicity, that $x_0$ is the leftmost point in the support of the solution. Using H\"older's inequality and assuming that the support of $u$ remains to the right of $x_0$, one obtains from the monotonicity formula applied with $x_0-\delta$ in place of $x_0$
\begin{align*}
&\frac{\mathrm{d}}{\mathrm{d}t} \int_{\mathbb{R}} u^{1+\alpha} |x-x_0+\delta|^{\gamma} \,\d x
\\&
\geq c \delta^{-\frac{(\gamma+1)n}{(1+\alpha)}-4} \bigg( \int_{\mathbb{R}} u^{1+\alpha} |x-x_0+\delta|^{\gamma} \,\d x\bigg)^{\frac{1+\alpha+n}{1+\alpha}}.
\end{align*}
This implies finite-time blowup of $\int_{\mathbb{R}} u^{1+\alpha}(\cdot,t) |x-x_0+\delta|^{\gamma} \,\d x$ and thereby a contradiction to the assumption that the support of $u(\cdot,T)$ remains to the right of $x_0$ as soon as
\begin{align*}
T\geq C \delta^{\frac{(1+\gamma)n}{(1+\alpha)}+4} \bigg(\int_{\mathbb{R}} u_0^{1+\alpha} |x-x_0+\delta|^{\gamma} \,\d x\bigg)^{-\frac{n}{(1+\alpha)}},
\end{align*}
so, in particular, as soon as
\begin{align*}
T \geq C \bigg(\delta^{-4(1+\alpha)/n} \dashint_{(x_0,x_0+\delta)} u_0^{1+\alpha} \,\d x\bigg)^{-n/(1+\alpha)}.
\end{align*}
The problem for ``concentrated'' initial data like \eqref{ConcentratedInitialData} is that the integral on the right-hand side of the previous formula is much smaller than suggested by the relation
\begin{align*}
\dashint_{(x_0,x_0+\delta)} u_0^{1+\alpha} \,\d x \sim \Bigg(\dashint_{(x_0,x_0+\delta)} u_0 \,\d x\Bigg)^{1+\alpha}
\end{align*}
which would be valid for initial data like $u_0(x)\sim (x-x_0)_+^\beta$.

The proof of our sharp sufficient condition \eqref{ConditionFreeBoundaryPropagation1d} for instantaneous forward motion of the free boundary is based on the following idea: If initially some amount of mass is present in the interval $(x_0,x_0+\delta)$, then there are basically two options -- either at least half of the mass remains near the interval $(x_0,x_0+\delta)$ up until at least time $T/2$, or at least half of the mass ``escapes'' from the vicinity of the interval before time $T/2$. In the former case, the monotonicity formula \eqref{MonotonicityFormula} entails a lower bound on $\int_{\mathbb{R}} u^{1+\alpha}(x,T/2) |x-x_0|^{\gamma} \,\d x$ by a simple application of H\"older's inequality, and it turns out that this lower bound is sufficient for the derivation of our result. In the latter case, a combination of the monotonicity formula \eqref{MonotonicityFormula} with a careful estimate based on testing the PDE \eqref{ThinFilmEquation} with a suitable smooth cutoff shows that motion of mass entails entropy production, again yielding a lower bound for $\int_{\mathbb{R}} u^{1+\alpha}(x,T/2) |x-x_0|^{\gamma} \,\d x$. In both cases, we then use the estimates of \cite{FischerARMA,FischerAHP}, starting at time $t_0=T/2$ instead of $t_0=0$, to conclude. The full argument is provided in Section~\ref{SectionProofNecessary}.

\subsection{Strategy for the upper bounds on free boundary propagation}
\label{SubSectionStrategyUB}

Our strategy for the derivation of upper bounds on free boundary propagation is based on the following concept: In the regime $n\in [2,3)$, we say that a solution to the thin-film equation $u$ is \emph{degenerate} on a parabolic cylinder $B_{r}(x_0)\times [0,T]$ if it satisfies both
\begin{subequations}
\begin{align}
\label{DegeneracyMass}
&\sup_{t \in (0,T)}\averageint_{B_r(x_0)} u \dx
\leq \varepsilon T^{-1/n} r^{4/n}
\end{align}
and
\begin{align}
\label{DegeneracyEnergy}
&\sup_{t \in (0,T)} \averageint_{B_r(x_0)} \frac{t^\beta}{T^\beta} |\nabla u|^2 \dx + \int_0^T \averageint_{B_r(x_0)} \frac{t^\beta}{T^\beta} \Big(\Big|\nabla u^{\frac{n+2}{6}}\Big|^6 + u^n |\nabla \Delta u|^2 \Big) \dx \dt
\\&~~~~~~~~~~~~~~~~~~~~~~~~~~~~~~~~~~~~~~~~~~~~~~~~~~~~~~
\nonumber
\leq \varepsilon^\delta T^{-2/n} (r^{4/n-1})^2
\end{align}
\end{subequations}
for some appropriately chosen constants $\varepsilon=\varepsilon(d,n)>0$ and $\delta=\delta(d,n)>0$ and some suitably chosen $\beta\in (0,1)$. In the regime $n \in (1,2)$, we use a closely related ansatz, which replaces the degeneracy condition in terms of the energy \eqref{DegeneracyEnergy} by a corresponding condition in terms of a (localized) entropy, see \eqref{entropy2} below for details. In the remainder of this exposition, we shall focus only on the case $n\in [2,3)$.

The central idea of our proof is to show that -- provided that the initial data also satisfy a degeneracy condition of the type \eqref{ConditionFreeBoundaryPropagation1dFinite} -- the degeneracy of $u$ on a parabolic cylinder $B_{r}(x_0)\times [0,T]$ implies the degeneracy of $u$ on the spatially smaller parabolic cylinder $B_{r/2}(x_0)\times [0,T]$ with the same time horizon $T$. Propagating the degeneracy down to $r\rightarrow 0$, this essentially shows $u(x_0,t)=0$ for $t\leq T$.

The general spirit of the proof is inspired by the approach of \cite{DalPassoGiacomelliGruen,GruenWTWS, GiacomelliGruen}, one difference being that in our formulation the iteration \`{a} la Stampacchia present in \cite{DalPassoGiacomelliGruen,GruenWTWS, GiacomelliGruen} is done essentially explicitly by the propagation of degeneracy. However, the key difference of our approach to \cite{DalPassoGiacomelliGruen,GruenWTWS, GiacomelliGruen} is that the latter is formulated in terms of the local energy only and does not keep track of the propagation of mass. This substantially simplifies the estimates, but comes at the cost of formulating the degeneracy condition on the initial data in terms of the local energy $\averageint_{B_r(x_0)} |\nabla u_0|^2 \dx$, making it impossible to derive an optimal result. By keeping track of the propagation of mass via \eqref{DegeneracyMass}, we are able to eliminate the dependence on the initial energy by introducing a weight $(t/T)^\beta$ in the degeneracy condition for the energy \eqref{DegeneracyEnergy}.

The rough idea for the propagation of the first degeneracy condition \eqref{DegeneracyMass} is the following: Starting with degenerate initial data $u_0$ (in the sense that the quantity \eqref{QuantityWaitingTime} is finite), after choosing $T$ appropriately (depending on the size of the quantity \eqref{QuantityWaitingTime}) it suffices to control the possible influx of mass $u$ into the smaller ball $B_{r/2}(x_0)$ up to time $T$. The degeneracy properties \eqref{DegeneracyMass} and \eqref{DegeneracyEnergy} on a spatially larger parabolic cylinder in turn ensure that the influx of mass into the smaller ball $B_{r/2}(x_0)$ remains sufficiently limited up to time $T$; to see this, we test the PDE \eqref{ThinFilmEquation} with a weight and estimate the right-hand side carefully.

In order to propagate the second degeneracy condition \eqref{DegeneracyEnergy} which involves the energy, we cannot rely on the energy of the initial data, as the localized $H^1$ norms of the initial data do not need to reflect the degeneracy of the initial data near $x_0$ (recall for instance the counterexample \eqref{OscillatoryInitialData}). We instead rely on the regularization properties of the nonlinear fourth-order parabolic operator, reducing the problem to an estimate on the local mass. This idea is close in spirit to the consideration (for the thin-film equation on a bounded domain $\Omega$)
\begin{align*}
\frac{\d}{\d t} \int_{\Omega} |\nabla u|^2 \,\d x
\leq -c \int_\Omega |\nabla u^{\frac{n+2}{6}}|^6 \,\d x
\leq -c(\Omega) \bigg(\int_\Omega u \,\d x\bigg)^{n-4} \bigg(\int_\Omega |\nabla u|^2 \,\d x\bigg)^3
\end{align*}
where the first step is just the energy dissipation property, combined with Bernis-Gr\"un's inequality, and the second step is a simple application of H\"older's inequality. This estimate implies by an elementary ODE argument a bound of the form
\begin{align}
\label{EstimateInTermsOfEntropy}
\int_\Omega |\nabla u(\cdot,t)|^2 \,\d x
\leq C(\Omega) t^{-1/2} \bigg(\sup_{s\in [0,t]}\int_\Omega u(\cdot,s) \,\d x\bigg)^{2-n/2},
\end{align}
which is now independent of $\int_\Omega |\nabla u_0|^2 \,\d x$, but blows up for $t\rightarrow 0$. Note that the blowup near initial time is the reason for our choice of including the factor $t^\beta/T^\beta$ in our condition \eqref{DegeneracyEnergy}.
A result of the type \eqref{EstimateInTermsOfEntropy} has been the basis for the existence theory for the thin film equation with measure-valued initial data in \cite{WeakInitialTrace}.

As for our purposes a global estimate on the energy in terms of the mass like \eqref{EstimateInTermsOfEntropy}  is insufficient -- we rather need to estimate a localized energy -- , to show the degeneracy condition for the energy \eqref{DegeneracyEnergy} on the spatially smaller parabolic cylinder we additionally need to control the influx of energy into the smaller ball $B_{r/2}(x_0)$ suitably. It turns our that the latter may be achieved using the control on mass and energy provided by the assumptions \eqref{DegeneracyMass} and \eqref{DegeneracyEnergy} on the bigger cylinder. In total, we obtain the degeneracy \eqref{DegeneracyEnergy} on the smaller cylinder $B_{r/2}(x_0)\times [0,T]$ as a result of the degeneracies \eqref{DegeneracyMass} and \eqref{DegeneracyEnergy} on the bigger cylinder $B_{r}(x_0)\times [0,T]$.

\section{Proof of the necessary condition for the waiting time phenomenon}
\label{SectionProofNecessary}

We now provide the proof of the upper bounds on waiting times and the sharp sufficient criteria for instantaneous forward motion of the free boundary stated in Theorem~\ref{th:massnec}.

\begin{proof}[Proof of Theorem \ref{th:massnec}]
{Recall that the upper bounds on waiting times of \cite{FischerARMA} and \cite{FischerAHP} are formulated in terms of the initial entropy, which may provide a suboptimal bound in case of concentrated initial data. Nevertheless, they will form the base for our optimal result:}
	Let $T^\ast:=\inf\{t\geq 0:\supp u(\cdot, t)\cap (-\infty,x_0) = \emptyset\}$. {It is our goal to show that by time $T^*/2$, the entropy must have increased sufficiently for an application of the results of \cite{FischerARMA,FischerAHP}.}
	
	To this aim, fix $r>0$ and let $\psi_r:{\R}\to \R$ be a function
	supported in $B_r(x_0)$ and symmetric around $x_0$ such that $0\leq \psi_r\leq
	1$, $\phi_r'(x)\leq 0$ for $x>x_0$,	$\psi_r\equiv 1$ on $B_{r/2}(x_0)$, and $|\nabla\psi_r|\leq
	C r^{-1}$, $|D^2\psi_r|\leq C  r^{-2}$, $|D^3\psi_r|\leq C  r^{-3}$, $|D^4\psi_r|\leq
	C r^{-4}$. Let $\varphi = \psi_r^k$, with $k$ to be chosen later large enough. 
	In the proof of the theorem, we shall distinguish two cases:
	\begin{enumerate}
		\item[1. ] $\quad \displaystyle \int_{B_r(x_0)}u(x,t)~\varphi(x) \dx ~\ge~ \frac{1}{2} \int_{B_r(x_0)} u_0\varphi \dx$ \ for all $t \in \left(0,T^*/2\right)$,\\{i.\,e.\ at least half of the initial (weighted) mass in an $r$-neighborhood of $x_0$ remains there up to time $T^*/2$, or}
		\item[2. ]$\quad \displaystyle \int_{B_r(x_0)} u(x,t)~\varphi(x) \dx ~\le~ \frac{1}{2} \int_{B_r(x_0)} u_0\varphi \dx$ \ for some $t \in \left(0,T^*/2\right)$,\\{i.\,e.\ at least half of the initial (weighted) mass in an $r$-neighborhood of $x_0$ ``escapes'' from the $r$-neighborhood prior to time $T^*/2$.}
	\end{enumerate}
{We will show that both cases entail sufficient entropy production up to time $T^*/2$ such that an application of the upper bounds on waiting times from \cite[Theorem 1]{FischerARMA} respectively \cite[Theorem 3]{FischerAHP} starting at time $T^*/2$ yields an optimal upper bound on the waiting time.}

	\smallskip
	\noindent\emph{Case 1.} 
	By applying {the upper bound on waiting times from} \cite[Theorem 1]{FischerARMA} respectively \cite[Theorem 3]{FischerAHP} -- depending on the value of $n \in (2,3)$ -- starting at time $t_0=T^*/2$ instead of $t_0=0$, we obtain
	\begin{align*}
	\left(T^* - \frac{T^*}{2}\right) \le C r^{4+ \frac{n}{\alpha + 1}(1+\gamma)} \left(\int_{\R}u^{\alpha + 1}\left(x,T^*/2 \right) |x-x_0+r|^\gamma \dx \right)^{-\frac{n}{\alpha + 1}}
	\end{align*}
	for certain suitable $-1 < \alpha < 0$ and $\gamma < -1$.
	{The monotonicity formula \eqref{MonotonicityFormula} (which holds for the chosen values of $\alpha$ and $\gamma$, since the results of \cite[Theorem 1]{FischerARMA} respectively \cite[Theorem 3]{FischerAHP} are based on it) entails the lower bound on the entropy at time $T^*/2$
	\begin{align*}
	&\int_{\mathbb{R}} u^{1+\alpha}(x,\tfrac{T^*}{2}) |x-x_0+r|^{\gamma} \,\d x
	\\&
\geq c \int_0^{T^*/2} \int_{\mathbb{R}} u^{1+\alpha+n} |x-x_0+r|^{\gamma-4} + |\nabla u^{\frac{1+\alpha+n}{4}}|^4 |x-x_0+r|^{\gamma} \,\d x \, \d t.
    \end{align*}		
	Plugging this estimate into the previous inequality, we deduce}
	\begin{equation}
	\label{mon2}
	\begin{split}
	&
	r^{-\frac{4}{n}}
	\Bigg(
	\int_0^{T^*/2} \dashint_{B_r(x_0)} \left|\nabla
	u^\frac{\alpha + n + 1}{4}\right|^4 +r^{-4}u^{\alpha + n + 1}\dx\dt
	\Bigg)^\frac{1}{\alpha + 1}
	\leq C{T^\ast}^{-\frac{1}{n}}.
	\end{split}
	\end{equation}
	{We now show that the continued presence of mass in the $r$-neighborhood of $x_0$ entails entropy production:}
	Jensen's inequality yields
	\begin{align*}
	{T^*}^{-\frac{1}{n}} &\ge 	C
	r^{-\frac{4}{n}}
	\Bigg(r^{-4} \int_0^{T^*/2} \Bigg( \dashint_{B_r(x_0)} u \dx \Bigg)^{\alpha + n + 1}\dt \Bigg)^\frac{1}{\alpha + 1}.
	\end{align*}
	Using the assumption {of continued presence of mass}
	\begin{align*}
	\int_{B_r(x_0)}u(\cdot,t) \varphi \dx \ge \frac{1}{2} \int_{B_r(x_0)} u_0\varphi \dx 	\quad \text{ for } \ t \in \left(0,T^*/2\right),
	\end{align*}
	we obtain
	\begin{align*}
	{T^*}^{-\frac{1}{n}}
	&\ge
	C
	r^{-\frac{4}{n}}~{T^\ast}^{\frac{1}{\alpha + 1}}
	r^{-\frac{4}{\alpha + 1}} \left( \dashint_{B_r(x_0)}  u_0\varphi \dx\right)^\frac{\alpha + n + 1}{\alpha + 1}.
	\end{align*}
	This implies
	\begin{align*}
	{T^*}^{-\frac{\alpha + n + 1}{n(\alpha + 1)}} &\ge C r^{-\frac{4(1+\alpha +n)}{n(\alpha + 1)}}\left( \dashint_{B_r(x_0)}  u_0\varphi \dx\right)^\frac{\alpha + n + 1}{\alpha + 1},
	\end{align*}
	which directly yields the desired estimate
	 \begin{align*} T^* \le C\left(r^{-\frac{4}{n}} \dashint_{B_r(x_0)} u_0 \dx \right)^{- n}.
	 \end{align*}
	
	\medskip
	\noindent\emph{Case 2.}
	{If at least half of the (weighted) mass escapes from the $r$-neighborhood of $x_0$ before time $T^*/2$, we would like to make use of the weak formulation of the PDE \eqref{ThinFilmEquation} and the monotonicity formula \eqref{MonotonicityFormula} to show that motion of mass entails entropy production:}
	For a smooth cut-off function $\varphi$ and any $T \in (0, T^*/2)$, we have
	\begin{align*}
	&\int_{\R} u(x,T)\varphi \dx
	\\&
	=\int_{\R}u_0 \varphi \dx
	+\int_0^T \int_{\R}u^n\nabla \Delta u \cdot\nabla \varphi \dx\dt
	\\&
	=\int_{\R}u_0 \varphi \dx
	-\int_0^T \int_{\R}u^n D^2u :D^2 \varphi + nu^{n-1}\nabla u\cdot D^2u \cdot \nabla
	\varphi \dx\dt
	\\&
	=\int_{\R}u_0 \varphi \dx
	+\int_0^T \int_{\R}u^n \nabla u \cdot \nabla \Delta \varphi
	+n u^{n-1}\nabla u \cdot D^2\varphi \cdot \nabla u
	\dx\dt
	\\&~~~
	+\int_0^T \int_{\R}
	\frac{n}{2}u^{n-1}|\nabla u|^2 \Delta \varphi
	+\frac{n(n-1)}{2}u^{n-2}|\nabla u|^2\nabla u \cdot\nabla \varphi
	\dx\dt
	\\&
	\geq
	\int_{\R}u_0 \varphi \dx
	-C\int_0^T\int_{B_r(x_0)}
	u^\frac{n+1-3\alpha}{4} \left|\nabla u^\frac{\alpha + n + 1}{4}\right|^3
	|\nabla \varphi|
	\dx\dt
	\\&~~~
	-C\int_0^T\int_{B_r(x_0)}
	u^{n+1}\left(|\Delta^2 \varphi|+\frac{|D^2\varphi|^3}{|\nabla \varphi|^2}\right)
	\dx\dt
	.
	\end{align*}
	{In principle, one may be tempted to estimate the last two terms by the terms in the monotonicity formula \eqref{MonotonicityFormula} in order to prove that motion of mass entails entropy production, for instance by bounding $\smash{\int_0^T \int_{B_r(x_0)} } u^{n+1} \dx\dt$ by interpolating between $\smash{\int_0^{T^*/2} \int_{B_r(x_0)}} |\nabla u^{(1+\alpha+n)/4}|^4 \dx\dt$ (to gain spatial integrability; recall that $\alpha<0$) and $\sup_{t\in [0,T^*/2]} \int_{B_r(x_0)} u^{1+\alpha} \dx$ (to gain time integrability). However, due to the involved exponents this attempt fails (as $1+\alpha$ may be arbitrarily close to $0$, see \cite{FischerAHP}). Instead, we need to replace $\int_{B_r(x_0)} u^{1+\alpha} \dx$ in this interpolation argument by the weighted mass $\int_{\R}u \varphi \dx$ itself, resulting in technical difficulties and yielding (in total) a differential inequality (in integral form) for the weighted mass $\int_{\R}u \varphi \dx$.
	More precisely, we get} by H\"{o}lder's inequality
	\begin{align*}
	&\int_{\R}u(x,T)\varphi \dx
	\\&
	\geq
	\int_{\R}u_0 \varphi \dx
	-C\int_0^T 
	\left(
	\int_{B_r(x_0)} \varphi^{4-\eps} ~u^{n+1-3\alpha} \dx
	\right)^\frac{1}{4}
	\\&~~~\quad \times
	\left(\int_{B_r(x_0)}
	\frac{|\nabla \varphi|^\frac{4}{3}}{\varphi^\frac{4-\eps}{3}}
	 \left|\nabla u^\frac{\alpha + n + 1}{4}\right|^4
	+\left(\frac{|\Delta^2\varphi|^\frac{4}{3}}{\varphi^\frac{4-\eps}{3}}
	+\frac{|D^2\varphi|^4}{|\nabla \varphi|^\frac{8}{3}\varphi^\frac{4-\eps}{3}}\right)
	u^{\alpha + n + 1}
	\dx\right)^\frac{3}{4}\dt
	.
	\end{align*}
	From Lemma \ref{WeightedInterpolation}, it follows that
	\begin{align*}
	&\int_{\R}u(x,T)\varphi \dx
	\\&
	~~\geq
	\int_{\R}u_0 \varphi \dx
	-C
	\int_0^T\left(\int_{B_r(x_0)}  \left|\nabla u^\frac{\alpha + n + 1}{4}\right|^4
	+r^{-4}u^{\alpha + n + 1}\dx\right)^\frac{\vartheta(n+1-3\alpha)}{4(1+\alpha+n)}
	\\&~~\quad\times
	\left(\int_{B_r(x_0)} \varphi u
	\dx\right)^\frac{(1-\vartheta)(n+1-3\alpha)}{4}
	\\&~~\quad\times
	\left(\int_{B_r(x_0)}
	\frac{|\nabla \varphi|^\frac{4}{3}}{\varphi^\frac{4-\eps}{3}}
	 \left|\nabla u^\frac{\alpha + n + 1}{4}\right|^4
	+\left(\frac{|\Delta^2\varphi|^\frac{4}{3}}{\varphi^\frac{4-\eps}{3}}
	+\frac{|D^2\varphi|^4}{|\nabla \varphi|^\frac{8}{3}\varphi^\frac{4-\eps}{3}}\right)
	u^{\alpha + n + 1}
	\dx
	\right)^\frac{3}{4}\dt
	.
	\end{align*}
	{Recall that $\varphi:=\psi_r^k$, where $\psi_r:{\R^d}\to \R$ is} a function supported in $B_{r}(x_0)$ with $0\leq \psi_r\leq
	1$, $\psi_r\equiv 1$ on $B_{r/2}(x_0)$, and $|D^m \psi_r|\leq
	Cr^{-m}$ for $1\leq m\leq 4$. { By choosing} $k$ large enough (depending on $\eps$), we obtain
	\begin{align*}
	&\int_{\R^d}u(x,T)\varphi \dx
	\\&
	~~~\geq
	\int_{\R^d}u_0 \varphi \dx
	-C r^{-1}
	\int_0^T
	\left(\int_{B_r(x_0)}  \left|\nabla u^\frac{\alpha + n + 1}{4}\right|^4
	+r^{-4}u^{\alpha + n + 1}\dx\right)^{\frac{\vartheta(n+1-3\alpha)}{4(1+\alpha+n)}
		+\frac{3}{4}}
	\\&~~~~~~~~~~~~~~~~~~~~~~~~~~~~~~~~~~~\quad \times
	\left(\int_{B_r(x_0)} u\varphi \dx\right)^\frac{(1-\vartheta)(n+1-3\alpha)}{4}\dt
	,
	\end{align*}
	{which is the desired integral inequality for the weighted mass $\int_{\R^d}u \varphi \dx$.}
	Since the solution of the differential equation 
	 \begin{align*} \frac{\d}{\d t}z(t) = q(t) \cdot [z(t)]^{m}\end{align*}
	is given by 
	 \begin{align*} z(t) = \left(z(0)^{1-m} - (m-1) \int_0^t q(s)\,\d s \right)^{\frac{1}{1-m}},\end{align*}
	a comparison argument yields (note that we have $(1-\vartheta)(n+1-3\alpha)<4$ by Lemma \ref{WeightedInterpolation} and we have $\vartheta(n+1-3\alpha)/(1+\alpha+n)=(n-3\alpha)/(4+\alpha+n)<1$ by $\alpha>-1$)
	\begin{align*}
	&\left(\int_{\R^d}  u(x,T) \varphi \dx\right)^\frac{4-(1-\vartheta)(n+1-3\alpha)}{4}
	\\& \ 
	\geq
	\left(\int_{\R^d}  u_0 \varphi \dx\right)^\frac{4-(1-\vartheta)(n+1-3\alpha)}{4}
	-Cr^{-1}
	T^{\frac{1}{4}-\frac{\vartheta(n+1-3\alpha)}{4(1+\alpha+n)}}
	\\&~~~~~~~~~~~~~ \times
	\left(\int_0^T \int_{B_r(x_0)}  \left|\nabla u^\frac{\alpha + n + 1}{4}\right|^4
	+r^{-4}u^{\alpha + n + 1}\dx\dt
	\right)^{\frac{\vartheta(n+1-3\alpha)}{4(1+\alpha+n)}+\frac{3}{4}}
	~.
	\end{align*}
	By making use of {the upper bound on the waiting time \eqref{mon2} (which was obtained from \cite{FischerARMA,FischerAHP} and the monotonicity formula)}, we infer for $T\leq T^*/2$
	\begin{align*}
	&\left(\int_{\R^d}  u(x,T) \varphi
	\dx\right)^\frac{3-n+3\alpha+\vartheta(n+1-3\alpha)}{4}
	\\ &~~~
	\geq
	\left(\int_{\R^d}  u_0 \varphi
	\dx\right)^\frac{3-n+3\alpha+\vartheta(n+1-3\alpha)}{4} \\&~~~~~~\quad 
	-C
	T^\ast r^{-1}
	\left(C{T^\ast}^{-\frac{\alpha + n + 1}{n}}r^\frac{4(\alpha + 1)+n}{n}
	\right)^{\frac{3-n+3\alpha+\vartheta(n+1-3\alpha)}{4(1+\alpha+n)}
		+\frac{n}{\alpha + n + 1}}
	\\&~~~
	\geq
	\left(\int_{\R^d}  u_0 \varphi
	\dx\right)^\frac{3-n+3\alpha+\vartheta(n+1-3\alpha)}{4}
	\\&~~~
	\quad -C
	{T^\ast}^{-\frac{3-n+3\alpha+\vartheta(n+1-3\alpha)}{4n}}
	\\&~~~~~~\quad \quad \times
	\left(r^{\frac{(4+n)(1+\alpha+n)}{n}-4-\alpha-n}
	\right)^{\frac{3-n+3\alpha+\vartheta(n+1-3\alpha)}{4(1+\alpha+n)}
		+\frac{n}{\alpha + n + 1}}
	\\&~~~
	\geq
	\left(\int_{\R^d}  u_0 \varphi
	\dx\right)^\frac{3-n+3\alpha+\vartheta(n+1-3\alpha)}{4}
	\\&~~~
	\quad -C
	{T^\ast}^{-\frac{3-n+3\alpha+\vartheta(n+1-3\alpha)}{4n}}
	r^\frac{(4+n)(3-n+3\alpha+\vartheta(n+1-3\alpha))}{4n}
	~,
	\end{align*}
i.\,e.	
\begin{align*}  
	\int_{B_r(x_0)} u(x,T) \varphi \dx \geq \int_{B_r(x_0)} u_0 \varphi
\dx -C
{T^\ast}^{-\frac{1}{n}}
r^{\frac{4}{n} +1}
\end{align*}
for any $T\leq T^*/2$.
	Combining this lower bound with the assumption {of ``escape of mass''}
	\begin{align*}
	\int_{B_r(x_0)}u(\cdot,t)\varphi \dx \le \frac{1}{2} \int_{B_r(x_0)} u_0\varphi \dx
	\quad \text{ for some } \ t \in \left(0,T^*/2\right),
	\end{align*}
	we obtain the desired estimate
	\begin{align*}
	T^* &\le \left(C r^{-\frac{4}{n}} \dashint_{B_r(x_0)} u_0 \varphi \dx \right)^{-n}  
	\le C\left( r^{-\frac{4}{n}} \dashint_{B_r(x_0)} u_0 \dx \right)^{-n}.
	\end{align*}
\end{proof}
In the proof of Theorem \ref{th:massnec}, we have used the following technical interpolation lemma.
\begin{lemma}
	\label{WeightedInterpolation}
	Let $d=1$, $n \in (2,3)$, and $\alpha \in (-1, 0)$ satisfying $\alpha + n < 2$. Let $u\in L^1(\mathbb{R})$ be a nonnegative function such that $u^\frac{\alpha + n + 1}{4}\in
	W^{1,4}(\mathbb{R})$. Let $\varphi:\mathbb{R}\rightarrow [0,1]$ be a
	smooth cut-off 
	function which symmetric around $x_0$, monotone decreasing in $|x-x_0|$, and satisfies $0 \le \varphi \le 1$ as well as
	\begin{align*}
	\displaystyle \varphi(x) = \begin{cases} 
	1, \quad & x \in B_{r/2}\left(x_0\right), \\
	 \displaystyle 0, & x \in \R \setminus B_r(x_0), \end{cases}
	\end{align*}
	and $|\nabla \varphi|\leq C$.
	For any $0 <\eps\ll 1$ small enough (depending only on $\alpha$, $n$, and $d$), there exists a constant $C>0$ (depending also only on $\alpha$, $n$, and $d$) such that the estimate
	\begin{align*}
	\int_{B_r(x_0)}& \varphi^{4-\eps} u^{n+1-3\alpha} \dx
	\\
	&~~\leq
	C
	\left(\int_{B_r(x_0)}  \left|\nabla u^\frac{\alpha + n + 1}{4}\right|^4
	+r^{-4}u^{\alpha + n + 1}\dx\right)^\frac{\vartheta(n+1-3\alpha)}{\alpha + n + 1}
	\\&\qquad \times
	\left(\int_{B_r(x_0)} \varphi u
	\dx\right)^{(1-\vartheta)(n+1-3\alpha)}
	\end{align*}
	holds, 
	 where $\vartheta \in (0,1)$ is given by
	\begin{align*}
	\vartheta =  \frac{(\alpha + n + 1)(n-3\alpha)}{(n+1-3\alpha)(4+\alpha + n)}.
	\end{align*}
	Furthermore, $\vartheta$ satisfies $(1-\vartheta)(n+1-3\alpha)<4$.
\end{lemma}
\begin{proof}
	The Gagliardo-Nirenberg-Sobolev interpolation inequality (applied to
	$v:=u^\frac{\alpha + n + 1}{4}$ with $p=\frac{4(n+1-3\alpha)}{\alpha + n + 1}$,
	$m=4$, $q=\frac{4}{\alpha + n + 1}$) implies, for $ s \in (r/2, r)$,
	\begin{align*}
	\int_{B_s(x_0)} u^{n+1-3\alpha} \dx
	\leq~ &
	C
	\left(\int_{B_r(x_0)}  \left|\nabla u^\frac{\alpha + n + 1}{4}\right|^4
	+r^{-4}u^{\alpha + n + 1}\dx\right)^\frac{\vartheta(n+1-3\alpha)}{\alpha + n + 1}
	\\&\quad \times
	\left(\int_{B_s(x_0)} u \dx\right)^{(1-\vartheta)(n+1-3\alpha)}
	~,
	\end{align*}
	where
	\begin{align*}
	\vartheta = \frac{ \frac{\alpha + n + 1}{4(n+1-3\alpha)} - \frac{\alpha + n + 1}{4} }{\frac{1}{4} - \frac{1}{d} - \frac{\alpha + n + 1}{4}} = \frac{(\alpha + n + 1)(n-3\alpha)}{(n+1-3\alpha)(4+\alpha + n)}.
	\end{align*}
	It is immediate that $0<\vartheta<1$.
	Note also that the constant $C$ does not depend on $s\in (r/2,r)$. 
	
	Fix $S \in \left(r/2, r\right)$; choosing $s(h):=\min\left\{\sup\{|x|:\varphi(x)\geq h\},S\right\}$ and integrating with
	respect to $h$, we infer
	\begin{align*}
	\int_{B_S(x_0)}& \varphi u^{n+1-3\alpha} \dx
	=\int_0^1\int_{B_{s(h)}(x_0)} u^{n+1-3\alpha} \dx~\d h
	\\&
	\leq
	C\left(\int_{B_r(x_0)}  \left|\nabla u^\frac{\alpha + n + 1}{4}\right|^4
	+r^{-4}u^{\alpha + n + 1}\dx\right)^\frac{\vartheta(n+1-3\alpha)}{\alpha + n + 1}
	\\&~~~~~~~~~\quad\times
	\int_0^1\left(\int_{B_{s(h)}(x_0)} u \dx\right)^{(1-\vartheta)(n+1-3\alpha)}~\d h
	\\&
	\leq
	C\left(\int_{B_r(x_0)}  \left|\nabla u^\frac{\alpha + n + 1}{4}\right|^4
	+r^{-4}u^{\alpha + n + 1}\dx\right)^\frac{\vartheta(n+1-3\alpha)}{\alpha + n + 1}
	\\&~~~~~~~~~\quad\times
	\left(\int_{B_S(x_0)} u \dx\right)^{(1-\vartheta)(n+1-3\alpha)-1}
	\int_0^1\int_{B_{s(h)}(x_0)} u \dx~\d h
	\\&
	\leq
	C\left(\int_{B_r(x_0)}  \left|\nabla u^\frac{\alpha + n + 1}{4}\right|^4
	+r^{-4}u^{\alpha + n + 1}\dx\right)^\frac{\vartheta(n+1-3\alpha)}{\alpha + n + 1}
	\\&~~~~~~~~~\quad \times
	\left(\int_{B_S(x_0)} u \dx\right)^{(1-\vartheta)(n+1-3\alpha)-1}
	\int_{B_r(x_0)} \varphi u \dx
	~.
	\end{align*}
	Repeating this procedure, we get for $k=2$ and subsequently $k=3$ (as long as $(1-\vartheta)(n+1-3\alpha)>3$)
	\begin{align*}\int_{B_S(x_0)}& \varphi^k u^{n+1-3\alpha} \dx =\int_0^1\int_{B_{s(h)}(x_0)} \varphi^{k-1}u^{n+1-3\alpha} \dx\,\d h 	
	\\&\le \int_0^1 C\left(\int_{B_r(x_0)}  \left|\nabla u^\frac{\alpha + n + 1}{4}\right|^4
	+r^{-4}u^{\alpha + n + 1}\dx\right)^\frac{\vartheta(n+1-3\alpha)}{\alpha + n + 1} 
	\\&~~~~~~  \times
	\left(\int_{B_{s(h)}(x_0)} u \dx\right)^{(1-\vartheta)(n+1-3\alpha)-(k-1)}
	\left(\int_{B_r(x_0)} \varphi u \dx\right)^{k-1} \,\d h 
	\\&
	\leq
	C\left(\int_{B_r(x_0)}  \left|\nabla u^\frac{\alpha + n + 1}{4}\right|^4
	+r^{-4}u^{\alpha + n + 1}\dx\right)^\frac{\vartheta(n+1-3\alpha)}{\alpha + n + 1}
	\\&~~~~~~ \times
	\left(\int_{B_S(x_0)} u \dx\right)^{(1-\vartheta)(n+1-3\alpha)-k}
	\left(\int_{B_r(x_0)} \varphi u \dx\right)^k
	~.
	\end{align*}
	Set  $\lambda:=(1-\vartheta)(n+1-3\alpha)$ and $k = \lfloor \lambda \rfloor$; note that we have $1<\lambda<4$ and therefore $k\in \{1,2,3\}$. By making use of the estimates above, we get  
	\begin{align*}
	&\int_{B_r(x_0)} \varphi^{(k+1)-\eps} u^{n+1-3\alpha} \dx
	=\int_0^r |\nabla \varphi^{1-\eps}(s)| \int_{B_s(x_0)} \varphi^k
	u^{n+1-3\alpha} \dx\ds
	\\&~~~
	\leq
	C\left(\int_{B_r(x_0)}  \left|\nabla u^\frac{\alpha + n + 1}{4}\right|^4
	+r^{-4}u^{\alpha + n + 1}\dx\right)^\frac{\vartheta(n+1-3\alpha)}{\alpha + n + 1}
	\\&~~~~~~\quad\times
	\left(\int_0^r \left|\nabla \varphi^{1-\eps}(s)\right|
	\left(\int_{B_s(x_0)} u \dx\right)^{(1-\vartheta)(n+1-3\alpha)-k}\ds\right)
	\left(\int_{B_r(x_0)} \varphi u \dx\right)^k
	\\&~~~
	\leq
	C\left(\int_{B_r(x_0)}  \left|\nabla u^\frac{\alpha + n + 1}{4}\right|^4
	+r^{-4}u^{\alpha + n + 1}\dx\right)^\frac{\vartheta(n+1-3\alpha)}{\alpha + n + 1}
	\left(\int_{B_r(x_0)} \varphi u \dx\right)^k
	\\&~~~~~~\quad \times
	\int_0^r \left|\nabla
	\varphi^\frac{1-\lambda + k-\eps}{1-\lambda+k}\right|^{1-\lambda + k}
	|\nabla \varphi|^{\lambda-k}
	\left(\int_{B_s(x_0)} u \dx\right)^{\lambda-k}\ds
	\\&~~~
	\leq
	C\left(\int_{B_r(x_0)}  \left|\nabla u^\frac{\alpha + n + 1}{4}\right|^4
	+r^{-4}u^{\alpha + n + 1}\dx\right)^\frac{\vartheta(n+1-3\alpha)}{\alpha + n + 1}
	\\&~~~~~~\quad \times
	\left(\int_{B_r(x_0)} \varphi u
	\dx\right)^{k}
	\left(\int_0^r |\nabla \varphi(s)| \int_{B_s(x_0)} \varphi u
	\dx\ds\right)^{\lambda - k}
	\\&~~~~~~\quad \times
	\left(\int_0^r  \left|\nabla \varphi^\frac{1-\lambda + k-\eps}{1-\lambda + k}\right| \ds\right)^{1-\lambda+k}
	\end{align*}
	(where we used H\"{o}lder's inequality in the last step). This proves the lemma
since $\int_{-\infty}^\infty |\nabla \varphi^\frac{1-\lambda+k-\eps}{1-\lambda+k}| \ds=2.$
\end{proof}

\section{Proof of the sufficient condition for the waiting time phenomenon}
\label{SectionProofSufficient}

We now turn to the proof of the optimal lower bounds on waiting times.
We split the proof in two cases: In the regime of strong slippage, i.\,e. $n \in (1,2)$, the ``propagation of degeneracy argument''  is based on the interplay between a localized mass estimate and a time-weighted localized entropy estimate; on the other hand, in the regime of weak slippage, i.\,e. $n \in [2,3)$, we employ a localized mass estimate and a time-weighted localized energy estimate. 
{
As the case $n\in (1,2)$ is somewhat more technical -- due to the additional variable $\alpha$ -- , we begin with the case $n\in [2,3)$.
}

\subsection{The case of weak slippage $n\in [2,3)$.}

\vspace{2mm}
\begin{proof}[Proof of Theorem 	\ref{th:masssuf}, case $n \in [2,3)$]
In case $x_0\in \partial \supp u_0$, denote by $\mathcal{C}$ a cone with the same apex and orientation as the cone from the exterior cone condition but with half the opening angle.
Our main assumption \eqref{ConditionWaitingTime} entails the existence of some $\rho>0$ such that for any point $\tilde x_0\in B_\rho(x_0)$ (if $x_0\notin \supp u_0$) respectively for any point $\tilde x_0\in B_\rho(x_0)\cap \mathcal{C}$ (if $x_0\in \partial \supp u_0$) the estimate
\begin{align*}
\averageint_{B_r(\tilde x_0)} u_0 \dx \le C(d,n,\lambda) \kappa\, r^{\frac{4}{n}}
\end{align*}
holds for all $r>0$. In other words, for all points $\tilde x_0$ near $x_0$ respectively all points $\tilde x_0$ near $x_0$ in the smaller cone $\mathcal{C}$, the initial data $u_0$ satisfy a growth condition analogous to \eqref{ConditionWaitingTime}, just with a different constant $\kappa$.
	
We will prove that the assumption \eqref{ConditionWaitingTime} implies that for $T:=c \kappa^{-n}$ and for $R>0$ large enough the estimate
\begin{align}
\label{DecayWeak}
\averageint_0^{T}\averageint_{B_{\frac{R}{2^k}}(x_0)} u \dx \dt \le \sup_{t \in (0,T)}\averageint_{B_{\frac{R}{2^k}}(x_0)} u \dx \dt \leq C  T^{-1/n}   \left(\frac{R}{2^k}\right)^{4/n}
\end{align}
holds for any $k \in \N$.
In view of the previous discussion, the same implication (up to adjusting the constants) then holds for all points $\tilde x_0$ in a neighborhood of $x_0$, respectively for all $\tilde x_0$ near $x_0$ which belong to the cone $\mathcal{C}$.
Letting $k\rightarrow\infty$, this implies $u(\tilde x_0,t) = 0$ for almost all such points $\tilde x_0$ and almost every $t\in (0,T)$. Thus, in view of Definition~\ref{DefinitionWaitingTime} this implies the lower bound $T^*\geq c \kappa^{-n}$ on the waiting time $T^*$ at $x_0$.

\vspace{2mm}
\noindent \emph{Step 1.}
{Fix $R>0$ and abbreviate $r_k:= R/2^k$ for any integer $k \ge 1$.
In order to prove \eqref{DecayWeak}, we will in fact derive corresponding estimates on the local mass
\begin{align}
\label{Mass}
	M(k) &:= \sup_{t \in [0,T]}\int_{B_{r_k}(x_0)} u \dx.
\end{align}
Our assumption \eqref{ConditionWaitingTime} implies a suitable estimate on the initial mass in the ball $B_{r_k}(x_0)$, i.\,e.\ a suitable bound for $M(k)$ in the case $T=0$.
Unfortunately, due to the larger exponent $u^{n+1}$ and the derivatives appearing on the right-hand side of the thin-film equation \eqref{ThinFilmEquation}, it is not possible to use the PDE \eqref{ThinFilmEquation} (by testing with a smooth cutoff, see below) to directly estimate the influx of mass into a smaller ball directly in terms of the mass on a larger ball.
While a localized energy dissipation estimate would be sufficient to control the influx of mass -- as well as the influx of energy -- , our only assumption \eqref{ConditionWaitingTime} does not provide any control of the initial local energy $\int_{B_{r_k}} |\nabla u_0|^2 \dx$. We therefore use the regularizing effect of the thin-film equation, which allows us to derive a time-dependent upper bound on the energy $\int_{B_{r_k}} |\nabla u_0|^2 \dx$:
We consider the energy (and energy dissipation) with a time-dependent increasing weight $t^\beta$, leading to the definition
\begin{align}
\label{Energy}
	E(k) &:= \sup_{t \in (0,T)} \int_{B_{r_k}(x_0)} t^\beta|\nabla u|^2 \dx
\\&\qquad
\nonumber
+ \int_0^T \int_{B_{r_k}(x_0)} t^\beta \Big(\Big|\nabla u^\frac{n+2}{6}\Big|^6 + u^n |\nabla \Delta u|^2 \Big) \dx \dt. 
\end{align}
We will see that the time evolution of the weighted energy $\int_{B_{r_k}(x_0)} t^\beta|\nabla u|^2 \dx$ can in fact be controlled: An increase in weighted energy can only stem from an influx of energy or from the time derivative of the weight $t^\beta$. We will show that by making use of the dissipation term, the latter term can be controlled in terms of the mass $M(k-1)$. However, due to the weight $t^\beta$ the influx of weighted energy into the ball $B_{r_k}$ can no longer be controlled by the weighted energy (and dissipation) alone; instead, we need to use an interpolation argument between the weighted energy (and dissipation) and the mass $M(k-1)$, see Step~2 below.
}
	
In order to control the influx of mass and energy into the balls $B_{r_k}$, we will make use of the family of smooth cutoff functions $\varphi_{r_k} \in C^\infty_c(\mathbb{R}^d)$ with the properties $\supp(\varphi_{r_k}) \subset B_{r_k}(x_0)$, $0 \le \varphi_{r_k} \le 1$, \begin{align*}\varphi_{r_k}(x) = \begin{cases}
	1, & x \in B_{r_{k+1}}(x_0), \\
	0, & x \in \R^d \setminus B_{r_k}(x_0),
	\end{cases}\end{align*} and $\displaystyle |\nabla\varphi_{r_k}|\leq
	C\left(r_{k+1}\right)^{-1}$, $\displaystyle |D^2\varphi_{r_k}|\leq C\left(r_{k+1}\right)^{-2}$, $\displaystyle|D^3\varphi_{r_k}|\leq C\left(r_{k+1}\right)^{-3}$, $\displaystyle|D^4\varphi_{r_k}|\leq
	C\left(r_{k+1}\right)^{-4}$. 
	
	\vspace{2mm}
	\noindent
	\emph{Step 2. Time-weighted localized energy estimate.}
	{We now derive an estimate for the weighted energy and dissipation $E(k+1)$ as defined in \eqref{Energy} in terms of the mass $M(k)$ and the weighted energy/dissipation $E(k)$ on the larger ball $B_{r_k}(x_0)$.}
	Let $T>0$ and $\varphi\in C^\infty_c(\mathbb{R}^d)$ be a nonnegative smooth cut-off function. In the appendix (Theorem \ref{WeightedEnergy}), we prove for any energy-dissipating weak solution of the thin-film equation \eqref{ThinFilmEquation} with zero contact angle in the sense of Definition~\ref{DefinitionEnergyDissipatingWeakSolution} the time-weighted localized energy estimate
	\begin{equation}\label{TWeightedEnergy}
	\begin{split}
	&\int_{\mathbb{R}^d} t^\beta |\nabla u|^2 \varphi^6 \dx\bigg|_{0}^{T}
	+C\int_{0}^{T}
	\int_{\mathbb{R}^d} t^\beta \Big(\Big|\nabla u^\frac{n+2}{6}\Big|^6 + u^n |\nabla \Delta u|^2 \Big) \varphi^6 \dx\dt
	\\&~~~
	\leq
	C\int_{0}^{T}\int_{\mathbb{R}^d} t^\beta u^{n+2} \left(|\nabla \varphi|^6 + |D^2\varphi|^2 |\nabla \varphi|^2 \varphi^2 + |D^2\varphi|^3 \varphi^3\right) \dx\dt
	\\&\qquad
	+\beta \int_{0}^{T} t^{-1} \int_{\mathbb{R}^d}
	t^\beta |\nabla u|^2 \varphi^6 \dx \dt
	\end{split}
	\end{equation}
	for any $\beta \in (0,1)$.
	H\"{o}lder's inequality yields
	\begin{align*}
	\int_{\mathbb{R}^d} |\nabla u|^2 \varphi^6 \dx
	\leq
	C\left(\int_{\mathbb{R}^d} |\nabla u^\frac{n+2}{6}|^6 \varphi^6 \dx
	\right)^\frac{1}{3}
	\left(\int_{\mathbb{R}^d} u \varphi^6 \dx\right)^\frac{4-n}{3}
	\left(\int_{\mathbb{R}^d} \varphi^6 \dx\right)^\frac{n-2}{3}
	~.
	\end{align*}
	Hence we have
	\begin{align*}
&\int_{\mathbb{R}^d} t^\beta |\nabla u|^2 \varphi^6 \dx\bigg|_{0}^{T}
+C\int_{0}^{T} 
\int_{\mathbb{R}^d} t^\beta \Big(\Big|\nabla u^\frac{n+2}{6}\Big|^6 + u^n |\nabla \Delta u|^2 \Big) \varphi^6 \dx\dt
\\&\leq
C\int_{0}^{T}\int_{\mathbb{R}^d} t^\beta u^{n+2} \left(|\nabla \varphi|^6 + |D^2\varphi|^2 |\nabla \varphi|^2 \varphi^2 + |D^2\varphi|^3 \varphi^3\right) \dx\dt
\\& \quad
+C\int_{0}^{T} t^{\beta - 1}
\left(\int_{\mathbb{R}^d} \varphi^6 \dx\right)^\frac{n-2}{3}
\left(\int_{\mathbb{R}^d} u \varphi^6 \dx\right)^\frac{4-n}{3}
\left(\int_{\mathbb{R}^d} \left|\nabla u^{\frac{n+2}{6}}\right|^6 \varphi^6 \dx\right)^{\frac{1}{3}}
\dt.
\end{align*}
By applying H\"{o}lder's inequality again, we obtain (assuming $\beta >1/2$)
	\begin{align*}
	&\int_{\mathbb{R}^d} t^\beta |\nabla u|^2 \varphi^6 \dx\bigg|_{0}^{T}
	+C\int_{0}^{T} 
	\int_{\mathbb{R}^d} t^\beta \Big(\Big|\nabla u^\frac{n+2}{6}\Big|^6 + u^n |\nabla \Delta u|^2 \Big) \varphi^6 \dx\dt
	\\ &\le C\int_{0}^{T}\int_{\mathbb{R}^d} t^\beta u^{n+2} \left(|\nabla \varphi|^6 + |D^2\varphi|^2 |\nabla \varphi|^2 \varphi^2 + |D^2\varphi|^3 \varphi^3\right) \dx\dt
	\\& \qquad + C\left(\int_0^T \int_{\R^d}   t^\beta |\nabla u^\frac{n+2}{6}|^6 \varphi^6 \dx \right)^{\frac{1}{3}} \\& \qquad \qquad \times \left(\int_{0}^{T} t^{\beta - \frac{3}{2}}  \left(\int_{\mathbb{R}^d} \varphi^6 \dx\right)^\frac{n-2}{2}
	\left(\int_{\mathbb{R}^d} u \varphi^6 \dx\right)^\frac{4-n}{2}
	\dt \right)^{\frac{2}{3}}
	~.
	\end{align*}
From Young's inequality (applied with suitably chosen constants), it follows that 
	\begin{align*}
&\int_{\mathbb{R}^d} t^\beta |\nabla u|^2 \varphi^6 \dx\bigg|_{0}^{T}
+C\int_{0}^{T} 
\int_{\mathbb{R}^d} t^\beta \Big(\Big|\nabla u^\frac{n+2}{6}\Big|^6 + u^n |\nabla \Delta u|^2 \Big) \varphi^6 \dx\dt
\\ &\le C\int_{0}^{T}\int_{\mathbb{R}^d} t^\beta u^{n+2} \left(|\nabla \varphi|^6 + |D^2\varphi|^2 |\nabla \varphi|^2 \varphi^2 + |D^2\varphi|^3 \varphi^3\right) \dx\dt
\\& \qquad +  C\int_{0}^{T} t^{\beta - \frac{3}{2}}  \left(\int_{\mathbb{R}^d} \varphi^6 \dx\right)^\frac{n-2}{2}
\left(\int_{\mathbb{R}^d} u \varphi^6 \dx\right)^\frac{4-n}{2}
\dt 
~.
\end{align*}
	Choosing $\varphi = \varphi_{r_k}$, the previous inequality reduces to 
	\begin{align*}
	\sup_{t\in (0,T)}&
	\int_{B_{r_{k+1}}(x_0)} t^\beta |\nabla u|^2   \dx
	+C\int_0^T \int_{B_{r_{k+1}}(x_0)}
	t^\beta \Big(\Big|\nabla u^\frac{n+2}{6}\Big|^6 + u^n |\nabla \Delta u|^2 \Big)  \dx\dt
	\\&
	\leq
	C\left(\frac{R}{2^{k+1}}\right)^{-6} \int_0^T\int_{B_{r_{k}}(x_0) } t^\beta u^{n+2} 	\dx\dt
	\\& \qquad +  C\left(\frac{R}{2^{k+1}}\right)^{d\frac{n-2}{2}} \int_{0}^{T} t^{\beta - \frac{3}{2}}	\left(\int_{B_{r_{k}}(x_0) } u  \dx\right)^\frac{4-n}{2}
	\dt~.
	\end{align*}
	The Gagliardo-Nirenberg-Sobolev  interpolation inequality (applied to $v=u^{\frac{n+2}{6}}$, $p = 6$, $q=\frac{6}{n+2}$, $r=6$) yields
	\begin{align*}
	\int_{B_{r_k}(x_0)} u^{n+2} \dx
	\leq&
	C\left(\int_{B_{r_k}(x_0)} |\nabla u^\frac{n+2}{6}|^6 \dx\right)^{\mu} 
	\left(\int_{B_{r_k}(x_0)} u \dx\right)^{(1-\mu)(n+2)}
	\\ &+C\left(\frac{R}{2^{k+1}}\right)^{-d(n+1)}\left(\int_{B_{r_k}(x_0)} u \dx\right)^{n+2}
	\end{align*}
	with
	\begin{align*}
	\mu=\frac{d(n+1)}{dn+d+6}
	~.
	\end{align*}
	We thus obtain using also the definition \eqref{Energy}
	\begin{align*}
	&
	{E(k+1)}
	\\&
	\leq
	C\left(\frac{R}{2^{k+1}}\right)^{-6}\int_0^T t^{(1-\mu)\beta}
	\left(\int_{{B_{r_k}(x_0)}} t^\beta|\nabla u^\frac{n+2}{6}|^6 \dx\right)^\mu
	\left(\int_{{B_{r_k}(x_0)}} u \dx\right)^{(1-\mu)(n+2)}
	\dt
	\\&\qquad
	+C\left(\frac{R}{2^{k+1}}\right)^{-6}  \left(\frac{R}{2^{k+1}}\right)^{-d(n+1)}\int_0^T t^{\beta}
	\left(\int_{{B_{r_k}(x_0)}} u \dx\right)^{n+2}
	\dt
	\\& \qquad +C \left(\frac{R}{2^{k+1}}\right)^{d\frac{n-2}{2}} \int_{0}^{T} t^{\beta - \frac{3}{2}}  
	\left(\int_{B_{r_k}(x_0)} u  \dx\right)^\frac{4-n}{2}
	\dt~.
	\end{align*}
	This implies by H\"{o}lder's inequality{, using also the assumption $\beta>\frac{1}{2}$ and the definitions \eqref{Energy} and \eqref{Mass} }
		\begin{align*}
	\\
	{E(k+1)}
	\leq &
	C\left(\frac{R}{2^{k+1}}\right)^{-6} T^{1-\mu + \beta(1-\mu)} 
	E(k)^\mu 
	M(k)^{(n+2)(1-\mu)}
	\\&\quad
	+C\left(\frac{R}{2^{k+1}}\right)^{-6-d(n+1)} T^{\beta +1}
	M(k)^{n+2}
	\\& \quad +C \left(\frac{R}{2^{k+1}}\right)^{d\frac{n-2}{2}} T^{\beta - \frac{1}{2}} 
	M(k)^\frac{4-n}{2}
	\end{align*}
	i.\,e. (inserting the expression for $\mu$)
	\begin{align}
	\label{LocalizedEnergyWeak}
	{E(k+1)}
	\leq &
	C\left(\frac{R}{2^{k+1}}\right)^{-6} T^{1-\mu + \beta(1-\mu)} 
	 E(k)^{\frac{nd+d}{nd+d+6}}
	M(k)^{\frac{6(n+2)}{nd+d+6}}
	\\&\quad
	\nonumber
	+C\left(\frac{R}{2^{k+1}}\right)^{-6-d(n+1)} T^{\beta +1}
	 M(k)^{n+2}
	\\& \quad
	+C \left(\frac{R}{2^{k+1}}\right)^{d\frac{n-2}{2}} T^{\beta - \frac{1}{2}} 
	\nonumber
	 M(k)^\frac{4-n}{2}
	\\ =:&E_1+E_2+E_3.
	\nonumber
	\end{align}

	\vspace{2mm}\noindent
	\emph{Step 3. Weighted mass estimate.} {We now derive an estimate for the average mass as defined in \eqref{Mass}.} We have by choosing $\psi = \varphi_{r_k}$ in Definition~\ref{DefinitionEnergyDissipatingWeakSolution}d
	\begin{align*}
	\sup_{t \in (0,T)}&\int_{B_{r_{k+1}}(x_0)} u \dx
	\\
	\leq&
	\int_{B_{r_k}(x_0)} u_0 \dx
	+C\left(\frac{R}{2^{k+1}}\right)^{-1}\int_0^T\int_{B_{r_k}(x_0)}
	u^\frac{n}{2}|\nabla \Delta u| u^{\frac{n}{2}}
	\dx\dt.
	\end{align*}
	The Gagliardo-Nirenberg-Sobolev interpolation inequality \eqref{GNSInterpolation} (applied to $v=u^{\frac{n+2}{6}}$, $p = \frac{6n}{n+2}$, $q=\frac{6}{n+2}$, $r=6$) yields
	\begin{align*}
	\bigg(\int_{B_{r_k}(x_0)} u^n \dx\bigg)^{\frac{1}{2}}
	\leq& C\left(\int_{B_{r_k}(x_0)} |\nabla u^\frac{n+2}{6}|^6\dx\right)^\frac{\vartheta n }{2(n+2)}
	\left(\int_{B_{r_k}(x_0)} u \dx\right)^{\frac{(1-\vartheta)n}{2}}
	\\ &+C\left(\frac{R}{2^{k+1}}\right)^{-\frac{d(n-1)}{2}}\left(\int_{B_{r_k}(x_0)} u \dx\right)^\frac{n}{2},
	\end{align*}
	with
	\begin{align*}
	\vartheta=\frac{d(n+2)(n-1)}{n(dn+d+6)}
	~.
	\end{align*}
	Putting these two estimates together, we deduce
	\begin{align*}
	&M(k+1)=\sup_{t \in (0,T)}\int_{B_{r_{k+1}}(x_0)}  u \dx
	\\&\leq
	\int_{B_{r_k}(x_0)}   u_0  \dx
	\\&~~~
	+C\left(\frac{R}{2^{k+1}}\right)^{-1}\int_0^T
	\left(\int_{B_{r_k}(x_0)}  u \dx\right)^\frac{n(1-\vartheta)}{2}\left(\int_{B_{r_k}(x_0)}  |\nabla u^\frac{n+2}{6}|^6
	\dx\right)^{\frac{\vartheta n}{2(n+2)}}
	\\&\qquad\qquad\qquad\qquad\qquad\qquad
	\times \bigg(\int_{B_{r_k}(x_0)} u^n |\nabla \Delta u|^2 \dx\bigg)^\frac{1}{2} \dt
	\\&~~~
	+C\left(\frac{R}{2^{k+1}}\right)^{-1-\frac{d(n-1)}{2}} \int_0^T \Bigg(\int_{B_{r_k}(x_0)}  u \dx \Bigg)^\frac{n}{2} \Bigg(\int_{B_{r_k}(x_0)} u^n |\nabla \Delta u|^2 \dx \Bigg)^\frac{1}{2}\dt.
	\end{align*}
	Choosing $\beta<\frac{3+d}{3+nd}$ (note that this is possible in view of the only other condition $\beta>\frac{1}{2}$), this implies by H\"{o}lder's inequality and the formula for $\vartheta$ as well as the definitions \eqref{Mass} and \eqref{Energy}
	\begin{align}
	\label{LocalizedMassWeak}	
	M(k+1)
	&\leq
	\int_{B_{r_k}(x_0)}   u_0  \dx 
	\\&~~~ 
	\nonumber
	+C\left(\frac{R}{2^{k+1}}\right)^{-1} T^{\frac{3+d-\beta dn -3\beta}{dn+d+6}}
	 M(k)^\frac{3n + d}{dn+d+6}
	E(k)^{\frac{dn+3}{dn+d+6}}
	\\&~~~
	\nonumber
		+C\left(\frac{R}{2^{k+1}}\right)^{-1-\frac{d(n-1)}{2}}T^{\frac{1-\beta}{2}}
	 M(k)^\frac{n}{2}
	E(k)^\frac{1}{2}
	\\ &=: M_1+M_2+M_3.
	\nonumber
	\end{align}
	
	\vspace{2mm}
	\emph{Step 4. Down-propagation of the degeneracy.}
	{The estimates \eqref{LocalizedEnergyWeak} and \eqref{LocalizedMassWeak} together form a Stampacchia-type inequality, however in contrast to the standard single-variable case now for two variables: \eqref{LocalizedEnergyWeak} and \eqref{LocalizedMassWeak} estimate $E(k+1)$ and $M(k+1)$ in terms of a higher power of $E(k)$ and $M(k)$, with prefactors that blow up for $k\rightarrow \infty$ in a controlled-exponential way. As the single-variable Stampacchia lemma is not applicable directly to our two-variable setting, we perform the Stampacchia iteration explicitly:}
	We want to prove {by induction} that for $R>0$ chosen large enough and for $T:=c\kappa^{-n}$, for every $k \in \N$ the bounds
	\begin{subequations}
	\label{Goal2Weak} 
	\begin{align}
	&
	M(k)
	\le \eps~ T^{-1/n}   \left(\frac{R}{2^k} \right)^{4/n+d}, \\
	&
	E(k)
	\le \eps^\delta~ T^\beta ~ T^{-2/n} \left(\frac{R}{2^k} \right)^{8/n -2+d}\end{align}
	\end{subequations}
	hold. Here, $\eps, \delta>0$ are suitable constants that will be chosen below. The parameter $\beta>0$ is arbitrary within the bounds mentioned above. Note that the estimate \eqref{Goal2Weak} will immediately imply our desired estimate \eqref{DecayWeak}. {Observe furthermore that the scaling of the estimates \eqref{Goal2Weak} is chosen precisely such that they reflect whether or not the function $u$ is ``degenerate'' on the parabolic cylinder $B_{r_k}(x_0)\times [0,T)$ (in the sense of having an average mass density of $r_k^{4/n}$ and a corresponding Dirichlet energy density $r_k^{8/n-2}$).}

	{To start the induction, it is easy to check that the estimate \eqref{Goal2Weak} holds for $k=1$} provided that we fix $R>0$ large enough. Indeed,
	\begin{align*}
	{ M(1)=}
	\sup_{t \in (0,T)}\int_{B_{r_1}(x_0)} u \dx \le \eps T^{-1/n}   \left(\frac{R}{2}\right)^{\frac{4}{n} + d}\end{align*} follows from the fact that 
	\begin{align}
	\label{BoundOverallMass}
	\int_{\R^d} u \dx = \int_{\R^d} u_0 \dx < \infty.
	\end{align}
	On the other hand, to prove 
	\begin{align*}
	&
	{ E(1)=}
	\sup_{t \in (0,T)} \int_{B_{r_1}(x_0)} t^\beta|\nabla u|^2 \dx + \int_0^T \int_{B_{r_1}(x_0)} t^\beta \Big(\Big|\nabla u^\frac{n+2}{6}\Big|^6 + u^n |\nabla \Delta u|^2 \Big) \dx \dt \\ &\ \qquad \le \eps~ T^\beta~ T^{-2/n}  \left(\frac{R}{2}\right)^{\frac{8}{n} - 2 + d}
	\end{align*}
	we argue as follows. By the property of finite speed of propagation for the solutions to the thin-film equation (see \cite[Theorem 1.3]{GruenOptimalRatePropagation}), there exists a ball $B_{\bar R}(x_0)$ that contains $\supp u(\cdot, t)$ for $t \in [0,T)$. We consider a smooth cut-off function $\varphi$ such that $\supp \varphi \subset B_{R}(x_0)$, $0 \le \varphi \le 1$, and $\varphi \equiv 1$ in $B_{\bar R}(x_0)$. Then, from the weighted energy estimate it follows that 
	\begin{align*}
	{E(1)=}\sup_{t\in (0,T)}&
	\int_{B_{r_1}(x_0)} t^\beta |\nabla u|^2 \dx
	+C\int_0^T \int_{B_{r_1}(x_0)}
	t^\beta \Big(\Big|\nabla u^\frac{n+2}{6}\Big|^6 + u^n |\nabla \Delta u|^2 \Big)  \dx\dt
	\\&\le C \int_{0}^{T} t^{\beta - \frac{3}{2}}  \left(\int_{B_{ \bar R}(x_0)} \varphi  \dx\right)^\frac{n-2}{2}
	\left(\int_{B_{ \bar R}(x_0)} u   \dx\right)^\frac{4-n}{2}
	\dt~
	\\&\le C~T^{\beta - \frac{1}{2}} \bar{R}^{\, d\frac{n-2}{2}} \Vert u_0 \Vert_{L^1(\R^d)}^{\frac{4-n}{2}}~.
	\end{align*}
	In view of \eqref{BoundOverallMass}, this implies the claim if we choose $R\geq\bar R$ large enough.

Having proved the base step of the induction, we now show that the bounds are propagated down to smaller scales: Assuming 
\begin{subequations}
\begin{align} \label{hype1}
&
{M(k)}
\le  \eps~T^{-1/n}   \left(\frac{R}{2^{k}}\right)^{\frac{4}{n}+d},
\\ \label{hype2}
&
{E(k)}
\le \eps^\delta~ T^\beta~T^{-2/n}    \left(\frac{R}{2^{k}}\right)^{\frac{8}{n} - 2+d},
\end{align}
\end{subequations}
we want to show that
\begin{subequations}
\begin{align}\label{mass22}
&
{M(k+1)}
\le  \eps~T^{-1/n}    \left(\frac{R}{2^{k+1}}\right)^{\frac{4}{n}+d},
\\ \label{energy22}
&
{E(k+1)}
\le \eps^\delta~ T^\beta~T^{-2/n}    \left(\frac{R}{2^{k+1}}\right)^{\frac{8}{n} - 2+d}.
\end{align}
\end{subequations}
\smallskip

Plugging the induction hypothesis \eqref{hype1}-\eqref{hype2} as well as the assumption \eqref{ConditionWaitingTime} into the localized energy and mass estimates \eqref{LocalizedEnergyWeak} and \eqref{LocalizedMassWeak}, we obtain 
	\begin{align*}
	&E_1 \le \dotuline{C~\eps^{\frac{6 (n+2 -\delta)}{dn + d+6}}}~\eps^{\delta}~T^{\beta}~T^{-\frac{2}{n}}\left(\frac{R}{2^{k+1}}\right)^{\frac{8}{n}-2+d} \\
	&E_2 \le \dotuline{C~\eps^{n+2-\delta}}~\eps^{\delta}~T^{\beta}~T^{-\frac{2}{n}}\left(\frac{R}{2^{k+1}}\right)^{\frac{8}{n}-2+d}   \\
	&E_3 \le \dotuline{C~\eps^{\frac{4-n-2\delta}{2}}}~\eps^{\delta}~T^{\beta}~T^{-\frac{2}{n}}\left(\frac{R}{2^{k+1}}\right)^{\frac{8}{n}-2+d}   \\
	\\
	&M_1 \le {C~\eps^{-1}\kappa~ T^{\frac{1}{n}}}~\eps~T^{-\frac{1}{n}}  \left(\frac{R}{2^{k+1}}\right)^{\frac{4}{n} +d} \\
	&M_2 \le \dotuline{C~\eps^{\frac{3n-dn-6+\delta dn + 3\delta}{dn+d+6}}}~\eps~T^{-\frac{1}{n}}  \left(\frac{R}{2^{k+1}}\right)^{\frac{4}{n} +d}\\
	&M_3 \le \dotuline{C~\eps^{\frac{n -2 +\delta }{2}}}~\eps~T^{-\frac{1}{n}}  \left(\frac{R}{2^{k+1}}\right)^{\frac{4}{n} +d}.
	\end{align*}
	Putting these estimates together, we conclude that 
	\eqref{mass22} and \eqref{energy22} hold if $\eps$ and  $\delta$ are chosen in a suitable way (i.\,e. $\eps$ small enough and $\delta$ in such a way that the exponents in the \dotuline{underlined factors} are positive, in particular $\delta<2-\frac{n}{2}$ but $\delta>(dn+6-3n)/(dn+3)$) and if we suppose that $T$ satisfies
		\begin{align*}
		C~\eps^{-1}~ T^{1/n}~ \kappa \le 1.
		\end{align*}
As a consequence, for such $T$ the estimates \eqref{Goal2Weak} hold.
\end{proof}

\subsection{The case of strong slippage $n\in (1,2)$.}

{
The proof in the regime $1<n<2$ is mostly analogous; however, the role of the weighted local energy is now played by a weighted local entropy, as Bernis-type inequalities are not known to hold for $d\geq 2$ and $n\approx 1$. Due to the additional variable exponent $1+\alpha$ in the entropy, the estimates are slightly more lengthy and technical.
}

\vspace{2mm}
\begin{proof}[Proof of Theorem 	\ref{th:masssuf}, case $n \in (1,2)$]
	We will prove the following statement: The assumption \eqref{ConditionWaitingTime} implies that for $T:=c\kappa^{-n}$ and for $R>0$ large enough the estimate   
	\begin{align}
	\label{GoalStrong}
	 \averageint_0^{T}\averageint_{B_{\frac{R}{2^k}}(x_0)} u \dx \dt \le \sup_{t \in (0,T)}\averageint_{B_{\frac{R}{2^k}}(x_0)} u \dx \leq C T^{-1/n}   \left(\frac{R}{2^k}\right)^{4/n}
	\end{align}
	holds for all $k \in \N$.
	To see that this implication entails our lower bound on waiting times, we refer to the discussion of the same issue in the case $n \in [2,3)$ provided at the beginning of the proof of Theorem~\ref{th:masssuf} in the case $n\in [2,3)$.
	
	\vspace{2mm}
	\noindent\emph{Step 1. Choice of test functions.} Fix $R>0$ and let $r_k:= 2^{-k} R$ for any $k \ge 1$. {Proceeding similarly to the case of weak slippage $n\in [2,3)$, we will prove estimates on the local mass
	\begin{align}
	\label{Mass2}
	M(k) &:= \sup_{t \in (0,T)}\int_{B_{r_k}(x_0)} u \dx
	\end{align}
	and the local weighted entropy and dissipation
	\begin{align}
	\label{Entropy}
	S(k) &:= \sup_{t \in (0,T)} \int_{B_{r_k}(x_0)} t^\beta u^{\alpha + 1} \dx + \int_0^T \int_{B_{r_k}(x_0)} t^\beta \left|\nabla u^{\frac{n+\alpha + 4}{4}}\right|^4 \dx \dt. 
	\end{align}
	}
	We will again make use of the family of cutoff functions $\varphi_{r_k} \in C^\infty_c(\mathbb{R}^d)$
	with $\supp(\varphi_{r_k}) \subset B_{r_k}(x_0)$, $0 \le \varphi_{r_k} \le 1$, \begin{align*}\varphi_{r_k}(x) = \begin{cases}
	1, &  x \in B_{r_{k+1}}(x_0), \\
	0, &  x \in \R^d \setminus B_{r_k}(x_0),
	\end{cases}\end{align*} and $\displaystyle |\nabla\varphi_{r_k}|\leq
	C(r_{k+1})^{-1}$, $\displaystyle |D^2\varphi_{r_k}|\leq C(r_{k+1})^{-2}$, $\displaystyle|D^3\varphi_{r_k}|\leq C(r_{k+1})^{-3}$, $\displaystyle|D^4\varphi_{r_k}|\leq
	C(r_{k+1})^{-4}$. 
	
	\vspace{2mm}
	\noindent\emph{Step 2. Time-weighted localized entropy estimate.} Let  $T>0$ and $\varphi \in C^\infty(\R^d \times [0,T])$ be a smooth nonnegative cut-off function. By arguing as in \cite[Theorem 3.1]{ThinViscous}, for weak solution to the thin-film equation \eqref{ThinFilmEquation} with zero contact angle in the sense of Definition~\ref{DefinitionWeakSolution} constructed with the approximation procedure in \cite{ThinViscous} it is possible to prove the time-weighted localized $\alpha$-entropy estimate
	\begin{equation} 
	\begin{split}
	&\sup_{t \in (0,T)}\int_{\R^d} \psi^{4}u^{\alpha+1} \dx \bigg|_{0}^{T}
	\\ &~~~~~~+ C\left( \int_{0}^{T}\int_{\R^d} \psi^{4}\left|D^2u^{\frac{\alpha+n+1}{2}}\right|^2 \dx \dt  + \int_{0}^{T}\int_{\R^d}\psi^4 \left|\nabla u^{\frac{\alpha+n+1}{4}} \right|^4 \dx \dt \right)
\\ &\quad \le C \int_{0}^{T} \int_{\{\psi>0\}} u^{\alpha+n+1}(|\nabla \psi|^4 + \psi^2|D^2 \psi|^2) \dx \dt
+\int_{0}^{T}\int_{\R^d}|\pt \psi| u^{\alpha + 1}\dx\dt
\end{split}
\end{equation}
for any $\alpha \in \left(\frac{1}{2} - n, 2-n\right)\setminus \{-1,0\}$, $\alpha>0$, and a.e. $T \ge 0$. This implies by taking $\psi = \varphi_{r_k}t^\beta$ (with $0 < \beta < 1$),  
\begin{align*}
\sup_{t \in (0,T)}&\int_{B_{r_{k+1}}(x_0)} t^\beta  u^{\alpha+1} \dx + C   \int_0^T\int_{B_{r_{k+1}}(x_0)} t^\beta \left|\nabla u^{\frac{\alpha+n+1}{4}} \right|^4 \dx \dt  \\ &\le  C \int_{0}^T \int_{B_{r_k}(x_0)} t^{\beta -1} u^{\alpha + 1}\dx \dt   
	\\ & \quad + C \left(\frac{R}{2^{k+1}} \right)^{-4}\int_{0}^T \int_{B_{r_k}(x_0)} t^\beta u^{\alpha+n+1} \dx \dt.
\end{align*}
The Gagliardo-Nirenberg-Sobolev interpolation inequality \eqref{GNSInterpolation} (applied to $v= u^{\frac{\alpha + n + 1}{4}}$ with $p = 4$, $q = \frac{4}{\alpha + n +1}$, $r = 4$) yields 
	\begin{align*}
	\int_{B_{r_k}(x_0)} u^{\alpha + n +1} \dx \le & C\left(\int_{B_{r_k}(x_0)} \left|\nabla u^{\frac{\alpha +n +1}{4}} \right|^4 \dx \right)^\sigma   \left( \int_{B_{r_k}(x_0)} u \dx\right)^{(\alpha + n +1)(1-\sigma)} \\ & \ +C\left( \frac{R}{2^{k+1}}\right)^{-d(\alpha + n)}\left(\int_{B_{r_k}(x_0)} u \dx  \right)^{\alpha + n + 1}
	\end{align*}
	with 
	\begin{align*}
	\sigma = \frac{d(\alpha + n)}{(d\alpha + dn + 4)}~.
	\end{align*}
	Again by the Gagliardo-Nirenberg-Sobolev interpolation inequality \eqref{GNSInterpolation} (applied to $v= u^{\frac{\alpha + n + 1}{4}}$ with $p = \frac{4(\alpha + 1)}{\alpha + n + 1}$, $q = \frac{4}{\alpha + n +1}$, $r = 4 $), we also have 
	\begin{align*}
	\int_{B_{r_k}(x_0)}  u^{\alpha +1} \dx 
	&\le  C\left(\int_{B_{r_k}(x_0)}  \left|\nabla u^{\frac{\alpha +n +1}{4}} \right|^4 \dx \right)^{\frac{\nu(\alpha + 1)}{\alpha + n + 1}}    \left(\int_{B_{r_k}(x_0)} u \dx\right)^{(\alpha + 1)(1-\nu)} \\ &\quad +C\left( \frac{R}{2^{k+1}}\right)^{-d\alpha }   \left(\int_{B_{r_k}(x_0)} u \dx  \right)^{\alpha   + 1}
	\end{align*}
	with 
	\begin{align*}
	\nu = \frac{d\alpha(1+\alpha +n)}{( d \alpha + dn + 4)(\alpha + 1)}~.
	\end{align*}
	Putting these considerations together {and using also the definition \eqref{Entropy},} we obtain
	\begin{align*}
{
S(k+1) =}	\sup_{t \in (0,T)}&\int_{B_{r_{k+1}}(x_0)} t^\beta  u^{\alpha+1} \dx + C   \int_0^T\int_{B_{r_{k+1}}(x_0)} t^\beta \left|\nabla u^{\frac{\alpha+n+1}{4}} \right|^4 \dx \dt  \\ \le~&  C\int_{0}^T  t^{\beta -1} \left(\int_{B_{r_k}(x_0)} \left|\nabla u^{\frac{\alpha +n +1}{4}} \right|^4 \dx  \right)^{\frac{\nu(\alpha + 1)}{\alpha + n + 1}} \\ &\qquad\times \left( \int_{B_{r_k}(x_0)} u \dx\right)^{(\alpha + 1)(1-\nu)}  \dt
	\\ 
	& +  C \left(\frac{R}{2^{k+1}}\right)^{-d\alpha} \int_{0}^T  t^{\beta -1} \left(\int_{B_{r_k}(x_0)} u \dx \right)^{\alpha + 1}  \dt  \\
	&+ C \left(\frac{R}{2^{k+1}} \right)^{-4}\int_{0}^T t^\beta \left(\int_{B_{r_k}(x_0)} \left|\nabla u^{\frac{\alpha + n +1}{4}} \right|^4\dx \right)^{\sigma}
	\\ &\qquad\qquad\qquad\qquad\qquad \times \left(\int_{B_{r_k}(x_0)} u\dx \right)^{(\alpha + n +1)(1-\sigma) }  \dt
	\\ &+ C \left(\frac{R}{2^{k+1}} \right)^{-4-d(\alpha + n)}\int_{0}^T t^\beta \left(\int_{B_{r_k}(x_0)} u\dx \right)^{\alpha + n +1 }   \dt.
	\end{align*}
	Finally, by using H\"{o}lder's inequality, we infer under the assumption $(1+\beta)(1-\frac{\nu(\alpha+1)}{\alpha+n+1})-1>0$ (which is satisfied for $\alpha>0$ small enough, the required smallness depending on $\beta>0$) using also \eqref{Mass2} and \eqref{Entropy}
	\begin{align*}
	S(k+1)
	\le&  C T^{(1+\beta)(1-\frac{(\alpha + 1)\nu}{\alpha + n + 1})  - 1}
	S(k)^{\frac{\nu(\alpha + 1)}{\alpha + n + 1}}
	M(k)^{(\alpha + 1)(1-\nu)}
	\\ 
	&+ \left(\frac{R}{2^{k+1}}\right)^{-d\alpha} T^\beta
	M(k)^{\alpha+1}
	\\
	&+ C \left(\frac{R}{2^{k+1}} \right)^{-4}  T^{(1-\sigma) (1+ \beta)} 
	S(k)^{\sigma}
	M(k)^{(\alpha + n +1)(1-\sigma) }
	\\ &+ C \left(\frac{R}{2^{k+1}} \right)^{-4-d(\alpha + n)} T^{\beta + 1} 
	M(k)^{\alpha + n +1},
	\end{align*}
	i.\,e.\ for $\alpha>0$ small enough we have (plugging in the definition of $\nu$ and $\sigma$)
	\begin{align}
	\label{EstimateEnergyStrong}
	&
	S(k+1)
	\\\nonumber
	 &\le  C T^{\beta - \frac{d\alpha}{d\alpha + dn +4} - \beta \frac{d\alpha}{d\alpha + dn +4}} {S(k)}^{\frac{d\alpha}{d\alpha + dn +4}}
	  {M(k)}^{\frac{dn +4 + 4d}{d\alpha + dn +4}} 
	  \\ \nonumber 
	&\quad+ \left(\frac{R}{2^{k+1}}\right)^{-d\alpha} T^\beta  {M(k)}^{\alpha + 1} 
	\\ \nonumber
	&\quad + C \left(\frac{R}{2^{k+1}} \right)^{-4}  T^{(1+\beta)\frac{4}{d\alpha + dn +4}} {S(k)}^{\frac{d\alpha + dn}{d\alpha + dn +4}}
	{M(k)}^{\frac{4(\alpha + n +1)}{d\alpha + dn +4}}
	\\ \nonumber &\quad + C \left(\frac{R}{2^{k+1}} \right)^{-4-d(\alpha + n)} T^{\beta + 1} {M(k)}^{\alpha + n +1 }
	\\&=: E_1+E_2+E_3+E_4.\nonumber
	\end{align}
	
	\vspace{2mm}
	\noindent \emph{Step 3. Localized mass estimate.} Starting from the weak formulation of the thin-film equation (see Definition~\ref{DefinitionWeakSolution}c), we obtain 
	\begin{align*}
	\int_{\R^d}& u(x,T)\varphi \dx
		\\&
		=\int_{\R^d} u_0 \varphi \dx + \int_0^T\int_{\Rd \cap  \{u>0\}} u^n \nabla u \cdot \nabla \Delta \varphi \dx \dt
	\\&\quad + n \int_0^T\int_{\Rd \cap  \{u>0\}} u^{n-1}\nabla u \cdot D^2\varphi \cdot \nabla u \dx \dt
	\\&\quad
	+ \frac{n}{2} \int_0^T\int_{\Rd \cap \{u>0\}} u^{n-1}|\nabla u|^2\Delta \varphi \dx \dt
	\\&\quad 
	+ \frac{n(n-1)}{2} \int_0^T\int_{\Rd \cap  \{u>0\}} u^{n-2}|\nabla u|^2\nabla u \cdot \nabla \varphi \dx \dt
		\\&
	\leq
	\int_{\R^d} u_0 \varphi \dx
	+C\int_0^T\int_{\R^d}
	u^\frac{n+1-3\alpha}{4} \left|\nabla u^\frac{\alpha + n + 1}{4}\right|^3
	|\nabla \varphi|
	\dx\dt
	\\&~~~
	+C\int_0^T\int_{\R^d}
	u^{n+1}\left(|\Delta^2 \varphi|+\frac{|D^2\varphi|^3}{|\nabla \varphi|^2}\right)
	\dx\dt
	~.
	\end{align*}
	Choosing $\varphi = \varphi_{r_k}$ as a test function, the previous inequality implies 
	\begin{align*}
	{M(k+1)}&=\sup_{t \in (0,T)}\int_{B_{r_{k+1}}(x_0)} u \dx
	\\&
	\leq
	\int_{B_{r_k}(x_0)} u_0 \dx
	+C\left(\frac{R}{2^{k+1}}\right)^{-1}\int_0^T\int_{B_{r_k}(x_0)}
	u^\frac{n+1-3\alpha}{4} \left|\nabla u^\frac{\alpha + n + 1}{4}\right|^3
	\dx\dt
	\\&~~~
	+C\left(\frac{R}{2^{k+1}}\right)^{-4}\int_0^T\int_{B_{r_k}(x_0)}
	u^{n+1}
	\dx\dt.
	\end{align*}
	By H\"{o}lder's inequality we obtain
	\begin{align*}
	&{M(k+1)}		\\
	&\le \int_{B_{r_k}(x_0)} u_0 \dx
	\\&~~~	+C\left(\frac{R}{2^{k+1}}\right)^{-1}\int_0^T
	\left(\int_{B_{r_k}(x_0)}	u^{n+1-3\alpha}\dx \right)^{\frac{1}{4}}  \left(\int_{B_{r_k}(x_0)} \left|\nabla u^\frac{\alpha + n + 1}{4}\right|^4 \dx \right)^{\frac{3}{4}}
	\dt
	\\&~~~
	+C\left(\frac{R}{2^{k+1}}\right)^{-4}\int_0^T\int_{B_{r_k}(x_0)}
	u^{n+1}
	\dx\dt.
	\end{align*}
	The Gagliardo-Nirenberg-Sobolev interpolation inequality \eqref{GNSInterpolation} (applied to $v= u^{\frac{\alpha + n + 1}{4}}$ with $p = \frac{4(n+1)}{\alpha + n +1}$, $q = \frac{4}{\alpha + n +1}$, $r = 4 $) yields
	\begin{align*}
	\int_{B_{r_k}(x_0)} u^{n+1 } \dx   &\le C \left(\int_{B_{r_k}(x_0)} \left|\nabla u^{\frac{\alpha + n + 1}{4}} \right|^{4} \dx \right)^{\frac{\vartheta (n+1)}{\alpha + n + 1}} \left(\int_{B_{r_k}(x_0)} u \dx \right)^{(1-\vartheta)(n+1)} \\ & \quad + C \left(\frac{R}{2^{k+1}}\right)^{-nd} \left(\int_{B_{r_k}(x_0)} u \dx \right)^{n+1}
	\end{align*}
	with 
	\begin{align*}\vartheta = \frac{nd(\alpha + n + 1)}{(nd+\alpha d + 4)(n+1)}.
	\end{align*}
	Note that for all $\alpha\in (0,2-n)$, all $d\in \{1,2,3\}$, and all $1<n<2$ we have $0<\vartheta<1$.
	By the Gagliardo-Nirenberg-Sobolev interpolation inequality \eqref{GNSInterpolation} (applied to $v= u^{\frac{\alpha + n + 1}{4}}$ with $p = \frac{4(n+1-3\alpha)}{\alpha + n + 1}$, $q = \frac{4}{\alpha + n +1}$, $r = 4 $), we also have
	\begin{align*}
	\left(\int_{B_{r_k}(x_0)} u^{n+1-3\alpha} \dx \right)^{\frac{1}{4}} &\le C \left(\int_{B_{r_k}(x_0)} \left|\nabla u^{\frac{\alpha + n + 1}{4}} \right|^{4} \dx \right)^{\frac{\mu(n+1-3\alpha)}{4(\alpha + n + 1)}}  \\ & \qquad \times \left(\int_{B_{r_k}(x_0)} u \dx \right)^{\frac{(1-\mu)(n+1-3\alpha)}{4}} \\ & \quad + C \left(\frac{R}{2^{k+1}}\right)^{-d \frac{(n-3\alpha)}{4}} \left(\int_{B_{r_k}(x_0)} u \dx \right)^{\frac{n+1-3\alpha}{4}},
	\end{align*}
	with
	\begin{align*}
	\mu = \frac{d(n-3\alpha)(\alpha + n + 1)}{(nd+\alpha d + 4)(n+1-3\alpha)}.
	\end{align*}
	Putting these considerations together, we obtain
	\begin{align*}
	{M(k+1)}\leq&
	\int_{B_{r_k}(x_0)} u_0 \dx
	\\ &+C\left(\frac{R}{2^{k+1}}\right)^{-1}\int_0^T 
	\left(\int_{B_{r_k}(x_0)} \left|\nabla u^{\frac{\alpha + n + 1}{4}} \right|^{4} \dx \right)^{\frac{\mu(n+1-3\alpha)}{4(\alpha + n + 1)}+\frac{3}{4}}
	\\ & \qquad\qquad\qquad\qquad\qquad \times \left(\int_{B_{r_k}(x_0)} u \dx \right)^{\frac{(1-\mu)(n+1-3\alpha)}{4}}\dt
	\\ & + C \left(\frac{R}{2^{k+1}}\right)^{-1-d \frac{(n-3\alpha)}{4}} \int_0^T \left(\int_{B_{r_k}(x_0)} u \dx \right)^{\frac{n+1-3\alpha}{4}}
	\\& \qquad\qquad\qquad\qquad\qquad\qquad \times \left(\int_{B_{r_k}(x_0)} |\nabla u|^4 \dx  \right)^{\frac{3}{4}} \dt 
	\\&
	+C\left(\frac{R}{2^{k+1}}\right)^{-4}\int_0^T  \left(\int_{B_{r_k}(x_0)} \left|\nabla u^{\frac{\alpha + n + 1}{4}} \right|^{4} \dx\right)^{\frac{\vartheta (n+1)}{\alpha + n + 1}}
	\\ & \qquad\qquad\qquad\qquad\qquad \times  \left(\int_{B_{r_k}(x_0)} u \dx \right)^{(1-\vartheta)(n+1)}\dt \\ & + C \left(\frac{R}{2^{k+1}}\right)^{-4-dn} \int_0^T  \left(\int_{B_{r_k}(x_0)} u \dx \right)^{n+1}\dt.
	\end{align*}
	This implies by H\"{o}lder's inequality, assuming that $1-(1+\beta)(\frac{\mu(n+1-3\alpha)}{4(\alpha+n+1)}+\frac{3}{4})>0$ and $\beta<\frac{1}{3}$ as well as $1-\frac{\vartheta(n+1)}{\alpha+n+1}-\beta\frac{\vartheta(n+1)}{\alpha+n+1}>0$,
	\begin{align*}
	{M(k+1)}\leq&
	\int_{B_{r_k}(x_0)} u_0 \dx \\ &+C\left(\frac{R}{2^{k+1}}\right)^{-1}
	T^{1-\frac{\mu(n+1-3\alpha)}{4(\alpha + n + 1)} - \frac{3}{4} - \beta \left( \frac{\mu(n+1-3\alpha)}{4(\alpha + n + 1)} + \frac{3}{4}\right) }  \\ & \qquad \times \left(\int_0^T  \int_{B_{r_k}(x_0)} t^\beta \left|\nabla u^{\frac{\alpha + n + 1}{4}} \right|^{4} \dx \dt \right)^{\frac{\mu(n+1-3\alpha)}{4(\alpha + n + 1)} +\frac{3}{4}}  \\ & \qquad \times \left(\sup_{t \in (0,T)} \int_{B_{r_k}(x_0)} u \dx \right)^{\frac{(1-\mu)(n+1-3\alpha)}{4}} \\ & + C \left(\frac{R}{2^{k+1}}\right)^{-1-d \frac{(n-3\alpha)}{4}} T^{\frac{1}{4} - \frac{3}{4}\beta} \left(\int_0^T \int_{B_{r_k}(x_0)} t^\beta \left|\nabla u^{\frac{\alpha + n + 1}{4}} \right|^{4} \dx \dt \right)^{\frac{3}{4}}  \\ & \qquad \times \left(\sup_{t \in (0,T)}\int_{B_{r_k}(x_0)} u \dx \right)^{\frac{n+1-3\alpha}{4}}
	\\&
	+C\left(\frac{R}{2^{k+1}}\right)^{-4} T^{1-\frac{\vartheta (n+1)}{\alpha + n +1} - \beta\frac{\vartheta (n+1)}{\alpha + n +1}} 	\\ &\qquad \times \left(\int_0^T  \int_{B_{r_k}(x_0)} t^\beta \left|\nabla u^{\frac{\alpha + n + 1}{4}} \right|^{4} \dx \dt \right)^{\frac{\vartheta(n+1)}{(\alpha + n + 1)}}  \\ & \qquad \times \left(\sup_{t \in (0,T)}\int_{B_{r_k}(x_0)} u \dx \right)^{(1-\vartheta)(n+1)} \\ & + C \left(\frac{R}{2^{k+1}}\right)^{-4-dn} T  \left(\sup_{t \in (0,T)}\int_{B_{r_k}(x_0)} u \dx \right)^{n+1}.
	\end{align*}
	Plugging in $\mu$ and $\vartheta$ as well as the definitions \eqref{Mass2} and \eqref{Entropy}, we deduce
	\begin{align}
	\label{EstimateMassStrong}
	&
	M(k+1)
	\\
	\nonumber
	&\leq
	\int_{B_{r_k}(x_0)} u_0 \dx
	\\
	\nonumber
	&~~~+C\left(\frac{R}{2^{k+1}}\right)^{-1}
	T^{\frac{\alpha d +1}{dn +\alpha d +4} - \beta \frac{dn+3}{dn +\alpha d +4}}
	S(k)^{\frac{dn+3}{dn +\alpha d +4}} 
	M(k)^{\frac{\alpha d + n +1 -3\alpha}{dn +\alpha d +4}}
	\\ 
	\nonumber
	&~~~ + C \left(\frac{R}{2^{k+1}}\right)^{-1-d \frac{(n-3\alpha)}{4}} T^{\frac{1}{4} - \frac{3}{4}\beta}
	S(k)^{\frac{3}{4}}
	M(k)^{\frac{n+1-3\alpha}{4}}
	\\&~~~
	\nonumber
	+C\left(\frac{R}{2^{k+1}}\right)^{-4} T^{\frac{\alpha d + 4}{dn +\alpha d +4} - \beta\frac{dn}{dn +\alpha d +4}}
S(k)^{\frac{dn}{dn +\alpha d +4}} 
M(k)^{\frac{\alpha d +4n+4}{nd +\alpha d +4}}
	 \\
	\nonumber&~~~ + C \left(\frac{R}{2^{k+1}}\right)^{-4-dn} T
	M(k)^{n+1}
	\\&=:M_1+M_2+M_3+M_4+M_5
	\nonumber
	\end{align}
	under the assumptions $(\alpha d+1)-\beta(dn+3)>0$, $\beta<\frac{1}{3}$, and $\alpha d+4-\beta dn>0$. Note that for $\beta<\frac{1}{9}$ these assumptions are satisfied.
	
	\vspace{2mm}
	\noindent
	\emph{Step 3. Down-propagation of the degeneracy.}
	{Arguing analogously to the case $n\in [2,3)$,}
	we want to prove {by induction} that for every $k \in \N$ the following estimates on the locally averaged mass and the locally averaged weighted entropy (and dissipation) hold: 
	\begin{subequations}
	\label{BoundsStrong}
	\begin{align}
	&M(k) \le \eps~ T^{-1/n}   \left(\frac{R}{2^k} \right)^{4/n+d}, \\
	&S(k) \le \eps^\delta~ T^\beta ~ T^{-(1+\alpha)/n} \left(\frac{R}{2^k} \right)^{4(\alpha +1)/n+d},
	\end{align}
	\end{subequations}
	where 
	$\eps, \delta >0$ are constants that will be chosen suitably below and where $\alpha$ and $\beta$ are arbitrary within the bounds given above.  Note that these estimates directly entail the desired result \eqref{GoalStrong} for all $k\in \mathbb{N}$.

	{As in the case $n\in [2,3)$, it is immediate to check that the estimates hold for $k=1$ provided that we fix $R>0$ large enough.} Indeed,
	\begin{align*}
	M(1)=
	\sup_{t \in (0,T)} \int_{B_{r_1}(x_0)} u \dx \le \eps ~ T^{-1/n}    \left(\frac{R}{2}\right)^{\frac{4}{n} + d}
	\end{align*} follows from the fact that 
	\begin{align*}
	\int_{\R^d} u \dx = \int_{\R^d} u_0 \dx < \infty.
	\end{align*}
	On the other hand, to prove 
	\begin{align*}
	S(1)=\sup_{t \in (0,T)} &\int_{B_{r_1}(x_0)} t^\beta u^{\alpha + 1} \dx + \int_0^T \int_{B_{r_1}(x_0)} t^\beta \left|\nabla u^{\frac{\alpha + n+1}{4}}\right|^4 \dx \dt \\ &\quad\le \eps^\delta~  T^\beta  ~ T^{-(\alpha+1)/n}  \left(\frac{R}{2}\right)^{\frac{4}{n}(\alpha + 1) + d}
	\end{align*}
	we argue as follows.  By the property of finite speed of propagation for the solutions to the thin-film equation (see \cite[Theorem 5.2]{ThinViscous}), 
	there exists a ball $B_{\bar{R}}(x_0)$  that contains $\supp u(\cdot, t)$ for $t \in [0,T)$. We consider a smooth cut-off function $\varphi$ such that $\supp \varphi \subset B_R(x_0)$, $0 \le \varphi \le 1$, and $\varphi \equiv 1$ in $B_{\bar R}$.  Then, from the weighted entropy estimate it follows that 
	\begin{align*}
	S(1)=
	\sup_{t \in (0,T)}&\int_{B_{r_1}(x_0)} t^\beta  u^{\alpha+1} \dx + C   \int_0^T\int_{B_{r_1}(x_0)} t^\beta \left|\nabla u^{{\tiny }\frac{\alpha+n+1}{4}} \right|^4 \dx \dt  \\ &\le  C T^{\beta - \frac{\beta(\alpha + 1)\nu}{\alpha + n + 1}  - \nu \frac{(\alpha + 1)}{\alpha + n + 1}}\left(\int_{0}^T   \int_{B_{ \bar R}(x_0)} t^\beta\left| \nabla u^{\frac{\alpha + n + 1}{4}} \right|^4  \dx \dt \right)^{\frac{\nu(\alpha + 1)}{\alpha + n + 1}}   \\ &\qquad \times \left(\sup_{t \in (0,T)}\int_{B_{ \bar R}(x_0)} u \dx \right)^{(\alpha + 1)(1-\nu)}   \\ 
	&\quad+ C \bar R^{-d\alpha} T^\beta  \left(\sup_{t \in (0,T)}\int_{B_{\bar R}(x_0)} u \dx \right)^{\alpha + 1} .
	\end{align*}
	Young's inequality yields with $\zeta:=\frac{\alpha+n+1}{\nu(\alpha+1)}$ and $\zeta'$ subject to $\frac{1}{\zeta}+\frac{1}{\zeta'}=1$
		\begin{align*}
S(1)=\sup_{t \in (0,T)}&\int_{B_{ \bar R}(x_0)} t^\beta  u^{\alpha+1} \dx + C   \int_0^T\int_{B_{ \bar R}(x_0)} t^\beta \left|\nabla u^{{\tiny }\frac{\alpha+n+1}{4}} \right|^4 \dx \dt  \\ &\le  C T^{\beta} T^{\zeta'(\beta - \frac{\beta(\alpha + 1)\nu}{\alpha + n + 1}  - \nu \frac{(\alpha + 1)}{\alpha + n + 1})-\beta} \left(\sup_{t \in (0,T)}\int_{B_{ \bar R}(x_0)} u \dx \right)^{(\alpha+1)(1-\nu)\zeta'}   \\ 
	&\quad + C \bar R^{-d\alpha} T^\beta  \left(\sup_{t \in (0,T)}\int_{B_{ \bar R}(x_0)} u \dx \right)^{\alpha + 1} .
	\end{align*}
	This implies the claim, if we choose $R$ and $\bar R$  large enough, since	\begin{align*}\int_{\R^d} u \dx = \int_{\R^d} u_0 \dx < \infty.\end{align*} 
	
	Having proved the base step of the induction, we now show that the bounds are propagated down to smaller scales: Assuming  
	\begin{align}\label{hypa1}
&M(k)
\le  \eps~T^{-1/n}    \left(\frac{R}{2^{k}}\right)^{\frac{4}{n} + d},
\\ \label{hypa2}
&S(k)
\le \eps^\delta~ T^\beta~T^{-(\alpha + 1)/n}    \left(\frac{R}{2^{k}}\right)^{\frac{4}{n}(\alpha + 1) + d},
\end{align}
we claim 
\begin{align}\label{mass2}
&
M(k+1)
\le  \eps~T^{-1/n}    \left(\frac{R}{2^{k+1}}\right)^{\frac{4}{n} + d},
\\ \label{entropy2}
&
S(k+1)
\le \eps^\delta~ T^\beta~T^{-(\alpha + 1)/n}    \left(\frac{R}{2^{k+1}}\right)^{\frac{4}{n}(\alpha + 1) + d}.
	\end{align}
	\smallskip
	
	Plugging the induction hypothesis \eqref{hypa1}-\eqref{hypa2} as well as the assumption \eqref{ConditionWaitingTime} into the localized entropy and mass estimates \eqref{EstimateEnergyStrong} and \eqref{EstimateMassStrong}, we obtain
\begin{align*}
	&E_1 \le \dotuline{C~\eps^{\frac{4+nd+4\alpha-\delta nd-4\delta}{dn+\alpha n +4}}}~\eps^\delta~ T^{\beta}~T^{-\frac{\alpha +1}{n}} \left(\frac{R}{2^{k+1}}\right)^{\frac{4}{n}(\alpha + 1) +d}~;\\
	&E_2 \le \dotuline{C~\eps^{\alpha+1-\delta}}~\eps^\delta~ T^{\beta}~T^{-\frac{\alpha +1}{n}} \left(\frac{R}{2^{k+1}}\right)^{\frac{4}{n}(\alpha + 1) +d}~;  \\
	&E_3 \le \dotuline{C~\eps^{\frac{4(\alpha + n +1 -\delta) }{nd+\alpha n + 4}}}~\eps^\delta~ T^{\beta}~T^{-\frac{\alpha +1}{n}} \left(\frac{R}{2^{k+1}}\right)^{\frac{4}{n}(\alpha + 1) +d}~;\\
	&E_4 \le \dotuline{C~\eps^{\alpha + n +1 -\delta}}~\eps^{\delta}~T^{\beta }~T^{-\frac{\alpha + 1}{n}} \left(\frac{R}{2^{k+1}}\right)^{\frac{4}{n}(\alpha + 1) +d}~; 
	\\
	\\
	&M_1 \le {C~\eps^{-1}~\kappa~ T^{\frac{1}{n}}}~\eps~T^{-\frac{1}{n}}  \left(\frac{R}{2^{k+1}}\right)^{\frac{4}{n} +d}~; 
	\\
	&M_2 \le \dotuline{C~\eps^{\frac{\delta nd + 3\delta -nd + n -3 -3\alpha }{nd+\alpha d +4} }} ~\eps~T^{-\frac{1}{n}}  \left(\frac{R}{2^{k+1}}\right)^{\frac{4}{n} +d}~; \\
	&M_3 \le \dotuline{C~\eps^{\frac{n-3-3\alpha +3\delta}{4}}}~\eps~T^{-\frac{1}{n}}  \left(\frac{R}{2^{k+1}}\right)^{\frac{4}{n} +d}~;\\
	&M_4 \le \dotuline{C~\eps^{\frac{4n-nd+\delta dn}{nd+\alpha d +4}}}~\eps~T^{-\frac{1}{n}}  \left(\frac{R}{2^{k+1}}\right)^{\frac{4}{n} +d}~;\\
	&M_5 \le \dotuline{C~\eps^{n}}~\eps~T^{-\frac{1}{n}}  \left(\frac{R}{2^{k+1}}\right)^{\frac{4}{n} +d}.
\end{align*}
Putting these estimates together, we conclude that 
\eqref{mass2} and \eqref{entropy2} hold if $\eps$,  $\delta$, and $\alpha$ are chosen in a suitable way (i.\,e. $\eps$ small enough and $\delta$ and $\alpha$ in such a way that the exponents in the \dotuline{underlined factors} are positive, for example $\delta:=1$ and $\alpha>0$ small enough) and if we suppose that $T$ satisfies
	\begin{align*}
	C~\eps^{-1}~ T^{1/n}~ \kappa \le 1.
	\end{align*}
	This completes the induction and shows \eqref{BoundsStrong} for all $k$.
\end{proof}

\vspace{5mm}
\appendix
\section{Auxiliary inequalities}
\label{AppendixInequalities}

\subsection{The monotonicity formula}

We recall the rigorous statement of the monotonicity formula \eqref{MonotonicityFormula}. Note that we only formulate the monotonicity formula for the range of values for which we have used in the proof of Theorem~\ref{th:massnec}; in fact, the monotonicity formula is valid for a much wider range of values of $n$, $\alpha$, and $\gamma$.
{\begin{theorem}[see \cite{FischerARMA,FischerAHP}]
\label{MonotonicityFormulaRigorous}
Let $d=1$ and let $n\in (2,3)$.
Let $u_0\in H^1(\mathbb{R})$ be compactly supported and nonnegative.
Let $u$ be an energy-dissipating weak solution to the thin-film equation with zero contact angle and initial data $u_0$ in the sense of Definition~\ref{DefinitionEnergyDissipatingWeakSolution}.
In case $n<\frac{32}{11}$, let $\alpha:=-\frac{11}{20}n+\frac{12}{20}$ and $\gamma:=-2$; in case $n\geq \frac{32}{11}$, let $\alpha:=\frac{1-n}{2}$ and $\gamma:=-\frac{11}{10}$.

Then for almost every $0<t_1<t_2<T$ the following statement holds: For any $x_0\notin \cup_{t\in [t_1,t_2]} \supp u(\cdot,t)$, the monotonicity formula
\begin{align*}
&\int_{\mathbb{R}} u^{1+\alpha}(x,t_2) |x-x_0|^\gamma \dx
\\&
\geq
\int_{\mathbb{R}} u^{1+\alpha}(x,t_1) |x-x_0|^\gamma \dx
\\&~~~~
+c\int_{t_1}^{t_2} \int_{\mathbb{R}} u^{1+\alpha+n} |x-x_0|^{\gamma-4} + \big|\nabla u^{(1+\alpha+n)/4}\big|^4 |x-x_0|^\gamma \dx\dt
\end{align*}
holds for some constant $c(n)>0$.
\end{theorem}
The proof of the monotonicity formula in case $n\in (2,\frac{32}{11})$ can be found in \cite[Proof of Theorem~1]{FischerARMA}, while the proof in the case $n\in [\frac{32}{11},3)$ is provided in \cite[Proof of Theorem~6]{FischerAHP}.
}

\subsection{Gagliardo-Nirenberg-Sobolev's interpolation inequality}
\label{SubSectionGagliardoNirenberg}

For the sake of completeness, we recall the version of Gagliardo-Nirenberg-Sobolev's interpolation inequality that has been used throughout the paper (see e.\,g.\ \cite[Proposition A.1]{DalPassoGiacomelliShishkov}).

\begin{theorem}[Gagliardo-Nirenberg-Sobolev's interpolation inequality]
	Let $\Omega \subset \R^d$ be an open bounded set with piecewise smooth boundary $\partial \Omega$. Let $0 < q < p$, $1 \le r \le \infty$, and $k \in \N$.  Let $v \in L^q(\Omega)$ such that $D^{k}v \in L^r(\Omega)$. Then there exist constants $C_1$ and $C_2$ (depending only on $\Omega$, $k$, $q$, and $r$) such that 
	\begin{align}\label{GNSInterpolation}
	\Vert v \Vert_{L^p(\Omega)} \le C_1\Vert D^k v\Vert_{L^r(\Omega)}^\vartheta  \Vert v \Vert_{L^q(\Omega)}^{1-\vartheta} +C_2\Vert v \Vert_{L^q(\Omega)},
	\end{align}
	where  \begin{align*} \vartheta := \frac{\frac{1}{q} - \frac{1}{p}}{\frac{1}{q} + \frac{k}{d} - \frac{1}{r}} \in (0,1).\end{align*}
	
	In addition, there exists a positive constant $C$ (depending on $r$, $m$, $q$ and $d$ and independent of $\Omega$) such that the following propositions hold true.
	\begin{enumerate}
		\item If $\Omega$ is either the  $d$-dimensional cube  $Q_\lambda(0)$ centered at the origin and of side-length $\lambda$ or the $d$-dimensional ball $B_{\lambda}(0)$ centered at the origin and of radius $\lambda$, then \eqref{GNSInterpolation} holds with 
		\begin{align*}C_1 = C \quad \text{ and } \quad C_2 = C \lambda^{-d\left(\frac{1}{q} - \frac{1}{p}\right)}.\end{align*}
		\item If $0 \le r_1 < r_2$, with $2r_1 > r_2$ if $d >1$, and $\Omega = B_{r_2}(0) \setminus B_{r_1}(0)$, then \eqref{GNSInterpolation} holds with 
		\begin{align*}C_1 = C \quad \text{ and } \quad C_2 = C\left(r_2-r_1 \right)^{-d\left(\frac{1}{q} - \frac{1}{p}\right)}.
		\end{align*}
		\item If $\Omega = \R^d$,  then \eqref{GNSInterpolation} holds with  
		\begin{align*}C_1 = C \quad \text{ and } \quad C_2 = 0.\end{align*}
	\end{enumerate}
\end{theorem}

\vspace{2mm}
\subsection{Bernis-Gr\"{u}n's weighted interpolation inequality}
\label{SubSectionBernisGruen}

We state Gr\"{u}n's weighted interpolation inequality (see \cite[Theorem III.1 and Corollary III.2]{CauchyGruen} and \cite{GruenBernis}), which was proved by Bernis in one space dimension (see \cite[Theorem 1]{BernisIntegral}) and plays an important role in handling the energy estimate for the thin-film equation. 

\begin{theorem}[Bernis-Gr\"{u}n's weighted interpolation inequality]
	Let $\Omega \subset \R^d$, with $2 \le d < 6$, be a bounded convex domain with a smooth boundary. Assume that a strictly positive function $u \in H^2(\Omega)$ satisfies 
	\begin{align*}\partial_\nu u\vert_{\partial \Omega} = 0 \qquad \text{ and } \qquad \int_{\Omega} u^n|\nabla \Delta u|^2 \dx < \infty,
	\end{align*}
	where $n \in \left(2- \sqrt{\frac{8}{8+d}}, 3\right)$. Let $\varphi \in C^1(\overline{\Omega})$ be a nonnegative function. Then there exists a positive constant $C$, which only depends on $d$ and $n$, such that 
	\begin{align*}
	&\int_{\Omega} \varphi^6 u^{n-4} |\nabla u|^6 \dx + \int_{\Omega} \varphi^6 u^{n-2} |D^2 u|^2 |\nabla u|^2 \dx \\ & \qquad + \int_{\partial \Omega} \varphi^6 u^{n-2}|\nabla u|^2 H(\nabla u, \nabla u) \dx \\ &~~~ \le C \left(\int_{\Omega} \varphi^6 u^n |\nabla \Delta u|^2 \dx + \int_{\{\varphi >0\}} u^{n+2} |\nabla \varphi|^6\dx\right),
	\end{align*} 
	where $H(\cdot, \cdot)$ is the second fundamental form of $\partial \Omega$.
	In particular, we have
\begin{align*}
&\int_{\Omega} \varphi^6|\nabla u^{\frac{n+2}{6}}|^6 \dx + \int_{\Omega} \varphi^6 |\nabla \Delta u^{\frac{n+2}{2}}|^2 \dx
\\& ~~~
\nonumber
\le C\left(\int_{\Omega} \varphi^6 u^n|\nabla \Delta u|^2 \dx + \int_{\{\varphi >0\}} u^{n+2}|\nabla \varphi|^6 \dx \right).
\end{align*}
\end{theorem}

\vspace{2mm}
\subsection{Weighted energy estimate}
\label{SubSectionWeightedEnergyEstimate}

Finally, we provide a proof of the energy estimate \eqref{TWeightedEnergy} by adopting a technique that resembles the one used in the proof of \cite[Lemma 1]{FischerARMA}. 

\begin{lemma}[Weighted energy estimate]
	\label{WeightedEnergy}
	Let $\Omega = \R^d$, $n \in \left(2-\sqrt{\frac{8}{8+d}}, 3\right)$,  and $u$ be an  energy-dissipating weak solution to the thin-film equation \eqref{ThinFilmEquation} with zero contact angle in the sense of Definition~\ref{DefinitionEnergyDissipatingWeakSolution}. Let $\psi\in
	C^2_c(\mathbb{R}^d)$ be a nonnegative weight function. Then we have
	\begin{align}\label{EnergyEstimateA}
	\nonumber&\int_{\mathbb{R}^d} \frac{1}{2}|\nabla u|^2\psi \dx\Bigg|_{t_1}^{t_2}
	-\int_{t_1}^{t_2}\int_{\mathbb{R}^d} \frac{1}{2}|\nabla u|^2\psi_t \dx\dt
	\\ &~~~=-\int_{t_1}^{t_2}\int_{\{u(\cdot, t)>0\}} u^n|\nabla \Delta u|^2\psi \dx\dt
	\\&
	\qquad \nonumber-\int_{t_1}^{t_2}\int_{\{u(\cdot, t)>0\}}u^n\nabla\Delta u\cdot
	\Big(\Delta u\nabla \psi+D^2u\cdot\nabla \psi+\nabla u\cdot D^2\psi\Big)
	\dx\dt
	\end{align}
	for a.e. $t_2\geq t_1\geq 0$ and a.e. $t_2\geq 0$ in case $t_1=0$.
\end{lemma}
\begin{remark}
	Let $\varphi \in C^\infty(\R^d)$. By applying Gr\"{u}n's weighted  inequality and Young's inequality, from the estimate above -- with $\psi := t^\beta \varphi^6$ and $\beta \in (0,1)$ -- we deduce \eqref{TWeightedEnergy}. 
	We observe that in \cite[Corollary 2.3]{GruenWTWS} inequality \eqref{TWeightedEnergy}  is proved for energy-dissipating weak solutions to the thin-film equation \eqref{ThinFilmEquation} in the sense of Definition~\ref{DefinitionEnergyDissipatingWeakSolution} as constructed  with the approximation procedure in \cite{CauchyGruen}.
\end{remark}

\begin{proof}[Proof of Lemma \ref{WeightedEnergy}]
	Let $\psi\in
	C^2_c(\mathbb{R}^d)$ be a nonnegative weight function and let $\rho_\delta\in C^\infty_c(\mathbb{R}^d)$ denote a standard mollifier with
	respect to space. Assume that $\dist(\supp \psi,  \partial\Omega \times (0,T))>\delta$.
	Using $-\nabla\cdot (\rho_\delta\ast (\psi~(\rho_\delta
	\ast \nabla u)))$ as a test function in the weak formulation of the thin-film
	equation  yields
	\begin{align}
	\nonumber
	&\int_{\mathbb{R}^d} \frac{1}{2}|\nabla \rho_\delta \ast u|^2 \psi \dx\bigg|_{t_1}^{t_2}
	-\int_{t_1}^{t_2}\int_{\mathbb{R}^d} \frac{1}{2}|\nabla \rho_\delta \ast u|^2 \psi_t
	\dx\dt
	\\&~~~
	\nonumber
	=-\int_{t_1}^{t_2}\int_{\mathbb{R}^d} \Big(\rho_\delta \ast u^n\nabla \Delta
	u\Big) \cdot \psi \nabla \Delta (\rho_\delta \ast u) \dx\dt
	\\&
	\label{A3FirstEquation}
	\qquad -\int_{t_1}^{t_2}\int_{\mathbb{R}^d} \Big(\rho_\delta \ast u^n\nabla \Delta u
	\Big) \cdot D^2(\rho_\delta \ast u)\cdot \nabla \psi \dx\dt
	\\&
	\nonumber
	\qquad-\int_{t_1}^{t_2}\int_{\mathbb{R}^d} \Big(\rho_\delta \ast u^n\nabla \Delta
	u\Big) \cdot \nabla \psi ~\Delta (\rho_\delta \ast u) \dx\dt
	\\&
	\nonumber
	\qquad-\int_{t_1}^{t_2}\int_{\mathbb{R}^d} \Big(\rho_\delta \ast u^n\nabla \Delta
	u\Big) \cdot D^2\psi \cdot \nabla (\rho_\delta \ast u) \dx\dt
	~.
	\end{align}
	We intend to pass to the limit as $\delta\rightarrow 0$. 
	Since $u \in L^\infty((0,T); H^1(\mathbb{R}^d))$, the terms on the
	left-hand side converge for a.e. $t_1$, $t_2$ to  \begin{align*} \int_{\mathbb{R}^d} \frac{1}{2}|\nabla u|^2\psi \dx\Bigg|_{t_1}^{t_2}
	-\int_{t_1}^{t_2}\int_{\mathbb{R}^d} \frac{1}{2}|\nabla u|^2\psi_t \dx\dt.\end{align*}
	
	By the definition of weak energy-dissipating solution, we have 
	 \begin{align*} \nabla u^{\frac{n+2}{6}} \in L^6((0,T); L^6(\mathbb{R}^d)) \quad \text{ and } \quad u^{\frac{n}{2}}\nabla \Delta u \in L^2((0,T); L^2(\mathbb{R}^d)).\end{align*}
	From Gagliardo-Nirenberg-Sobolev's embedding theorem and the property of conservation of mass, it follows that
	\begin{align*} u^{\frac{n+2}{6}} \in L^6((0,T); L^6(\mathbb{R}^d)).\end{align*} Therefore, 
	\begin{align*}\nabla u = \frac{6}{n+2} u^{\frac{4-n}{6}} \nabla u^{\frac{n+2}{6}} \in L^{n+2}\left((0,T); L^{n+2}(\R^d)\right).\end{align*}
	Moreover, 
	\begin{align*}
	u^{\frac{n}{2}} = \left(u^{\frac{n+2}{6}}\right)^{\frac{3n}{n+2}} \in L^{\frac{2(n+2)}{n}}\left((0,T);L^{\frac{2(n+2)}{n}}(\R^d)\right).
	\end{align*}
	In addition, due to $d \le 3$, by Sobolev's embedding theorem, for a.\,e.\ time $t\in [0,T]$ the function $u^{\frac{n+2}{6}}(\cdot,t)$ (and therefore $u(\cdot,t)$) is continuous.
	As a consequence, we have $\nabla \Delta u(\cdot,t) \in L^2_{loc}(\{u(\cdot,t)>0\})$ for a.e.
	 $t\in [0,T]$. From the regularity theory for elliptic operators, it follows that  $u(\cdot,t)\in H^3_{loc}(\{u(\cdot,t)>0\})$ for a.e.
	 $t\in [0,T]$. 
	Thus, on $\{u>0\}$, we immediately obtain pointwise convergence a.e. of the integrands on the
	right-hand side in formula \eqref{A3FirstEquation} in the limit $\delta\rightarrow 0$. 
	It remains to show that the integrands are dominated by integrable functions and to
	identify the pointwise limit on $\{u=0\}$ to infer convergence of the integrals.

	We start by studying the first integrand on the right-hand side of formula \eqref{A3FirstEquation}. Consider a smooth monotonous function $g$, with $0 \le g \le 1$, such that 
	$g \equiv 0$ for $x < 1/2$ and $g \equiv 1$ for $x>1$ and let
	 \begin{align}
	 \label{A3DefFbeta}
	 f_\beta(v) = \int_0^v g \left(\frac{s-\beta}{\beta} \right) \ds + \int_0^{2\beta} 1 - g \left(\frac{s-\beta}{\beta} \right) \ds,
	 \end{align}
	where $\beta > 0$. Note that this definition in particular entails $f_\beta(v)=v$ for any $v\geq 2\beta$ and $|f_\beta(v)-v|\leq 2 \beta$ for any $v\geq 0$.
	We may then rewrite
	\begin{align}
	\label{A3ThirdDerivative}
	\nabla\Delta (\rho_\delta \ast u)
	&=\nabla \Delta (\rho_\delta \ast (u-f_\beta(u))) 
	+\nabla\Delta (\rho_\delta \ast f_\beta(u))
	\\& =: I_{11} + I_{12}. \nonumber
	\end{align}
	We start by estimating $I_{11}$ as follows. 
	\begin{align*}
	&\Big|\nabla\Delta (\rho_\delta\ast (u-f_\beta(u)))\Big|(x_0)
	\leq C \delta^{-3} \dashint_{B_\delta(x_0)} |u-f_\beta(u)| \dx
	\\&~~~
	\leq C \delta^{-3} \beta \dashint_{B_\delta(x_0)} \chi_{\{u<2\beta\}} \dx
	\leq C\delta^{-3} \beta^{-\frac{n}{2}} \dashint_{B_\delta(x_0)}
	\Big((3\beta)^\frac{n+2}{6}-u^\frac{n+2}{6}\Big)_+^3 \dx
	\\&~~~
	\leq C\delta^{-3} \beta^{-\frac{n}{2}}
	\Bigg(\dashint_{B_\delta(x_0)}
	\left|u^\frac{n+2}{6}-\dashint_{B_\delta(x_0)} u^\frac{n+2}{6}(y) ~dy
	\right|^3\dx
	\\&~~~~~~~~~~~~~~~~~~~~~
	\qquad \quad +\left((3\beta)^\frac{n+2}{6}-\dashint_{B_\delta(x_0)} u^\frac{n+2}{6}(y) ~dy
	\right)_+^3
	\Bigg) ~.
	\end{align*}
	Choose \begin{align}
	\label{A3ChoiceBeta}
	\beta(x_0):=\nu\left(\dashint_{B_\delta(x_0)}
	u^n\dx\right)^\frac{1}{n}
	\end{align}
	with some $\nu>0$ to be fixed.
	Then, by the Poincar\'{e} inequality and the Sobolev embedding theorem, we obtain
	\begin{align*}
	&\Big|\nabla\Delta (\rho_\delta\ast (u-f_\beta(u)))\Big|(x_0)
	\\
	&~~~\leq C \beta^{-\frac{n}{2}} \dashint_{B_\delta(x_0)}
	\big|\nabla u^\frac{n+2}{6}\big|^3\dx
	\\&\qquad 
	+C\delta^{-3} \beta^{-\frac{n}{2}}
	\Bigg(C\nu^\frac{n+2}{6} \dashint_{B_\delta(x_0)} u^\frac{n+2}{6} \dx
	\\ &\qquad\qquad\qquad\qquad\quad +C\nu^\frac{n+2}{6}\delta\left(\dashint_{B_\delta(x_0)} |\nabla
	u^\frac{n+2}{6}|^3 \dx\right)^\frac{1}{3}
	-\dashint_{B_\delta(x_0)} u^\frac{n+2}{6} \dx \Bigg)_+^3
	~.
	\end{align*}
	Choosing $\nu>0$ small enough depending only on $n$ and $d$, we infer
	\begin{align}
	\label{A3BoundThirdDerivative}
	&\Big|\nabla\Delta (\rho_\delta\ast (u-f_\beta(u)))\Big|(x_0)
	\leq C \beta^{-\frac{n}{2}}
	\dashint_{B_\delta(x_0)} \Big|\nabla u^\frac{n+2}{6}\Big|^3 \dx
	~.
	\end{align}
	Secondly, we analyze $I_{12}$. We have
	\begin{align*}
	&\nabla\Delta (\rho_\delta \ast f_\beta(u))
	\\&~~~
	=\rho_\delta \ast \Big(f_\beta'(u)\nabla\Delta u+2f_\beta''(u)\nabla u\cdot D^2
	u +f_\beta''(u)\nabla u\Delta u+f_\beta'''(u)|\nabla u|^2\nabla u\Big),
	\end{align*}
	which implies (using the fact that $f''_\beta(v)=0$ and $f'''_\beta(v)=0$ for $v \notin [\beta, 2\beta]$ as well as the fact that $f'_\beta(v)=0$ for $v<\beta$ and the bounds $|f'_\beta|\leq C$, $|f''_\beta|\leq C \beta^{-1}$, and $|f'''_\beta|\leq C \beta^{-2}$)
	\begin{align}
	\label{A3BoundThirdDerivative2}
	&\Big|\nabla \Delta (\rho_\delta \ast f_\beta(u))\Big|
	\\&~~~~~
	\nonumber
	\leq \rho_\delta \ast
	\Big(\beta^{-\frac{n}{2}}|u^\frac{n}{2}\nabla\Delta u|
	+C\beta^{-\frac{n}{2}}|u^\frac{n-2}{2}\nabla u\otimes D^2u|
	+\beta^{-\frac{n}{2}}u^\frac{n-4}{2}|\nabla u|^3
	\Big)
	.
	\end{align}
	Summing up, we obtain
	\begin{align*}
	|\nabla \Delta (\rho_\delta\ast u)(x_0)|
	\leq C\beta^{-\frac{n}{2}} \dashint_{B_\delta(x_0)}\left(u^\frac{n}{2}|\nabla\Delta u|
	+u^\frac{n-2}{2}|\nabla u\otimes D^2 u|+ \left|\nabla u^\frac{n+2}{6}\right|^3\right)\dx
	~.
	\end{align*}
For a.e. $t \in (0,T)$, we have $\nabla u^{\frac{n+2}{6}} \in L^6(\R^d)$, $u^{\frac{n}{2}}\nabla \Delta u \in L^2(\R^d)$, and $u^{\frac{n-2}{2}} \nabla u \otimes D^2u \in L^2(\R^d)$.
Taking into account the estimate
	\begin{align*}
	|\rho_\delta \ast (u^n\nabla \Delta u) (x_0)|
	\leq
	C\left(\dashint_{B_\delta(x_0)} u^n|\nabla\Delta u|^2\dx\right)^\frac{1}{2}
	 \left(\dashint_{B_\delta(x_0)}u^n\dx\right)^\frac{1}{2}~,
	\end{align*}
	we infer by \eqref{A3ThirdDerivative}, \eqref{A3ChoiceBeta}, \eqref{A3BoundThirdDerivative}, and \eqref{A3BoundThirdDerivative2}
	\begin{align*}
	&\Big|(\rho_\delta \ast (u^n\nabla \Delta u)) (x_0)~\cdot~
	\psi(x_0)~~\nabla\Delta (\rho_\delta \ast u)(x_0)\Big|
	\\
	&~~~\leq
	C\left(\dashint_{B_\delta(x_0)}u^n|\nabla \Delta u|^2\dx\right)^\frac{1}{2} \psi(x_0) \\ &\qquad \qquad
	\times \dashint_{B_\delta(x_0)}\left(u^\frac{n}{2}|\nabla\Delta u|
	+u^\frac{n-2}{2}|\nabla u\otimes D^2 u|+ \left|\nabla u^\frac{n+2}{6}\right|^3\right)\dx
	\end{align*}
	This shows that the first integrand on the right-hand side of \eqref{A3FirstEquation} is dominated by a (space-time) integrable function and also implies that the pointwise limit of the
	integrand vanishes on $\{u(\cdot,t)=0\}$ for a.\,e.\ $t\in [0,T]$.
	
	\medskip
	For the other integrands on the right-hand side of \eqref{A3FirstEquation}, we use analogous arguments. Let us sketch the estimates for the second one.  Consider as before a smooth monotonous function $g$, with $0 \le g \le 1$, such that 
	$g \equiv 0$ for $x < 1/2$ and $g \equiv 1$ for $x>1$, and let $f_\beta$ be as in \eqref{A3DefFbeta}.
	We then write
	\begin{align*}
	D^2 (\rho_\delta \ast u)
	&=D^2 (\rho_\delta \ast (u-f_\beta(u))) 
	+D^2 (\rho_\delta \ast f_\beta(u))
	\\& =: I_{21} + I_{22}.
	\end{align*}
	We start by estimating $I_{21}$ as
	\begin{align*}
	&\Big|D^2 (\rho_\delta\ast (u-f_\beta(u)))\Big|(x_0)
	\leq C\delta^{-2} \dashint_{B_\delta(x_0)} |u-f_\beta(u)| \dx
	\\&~~~~~~~~~ \ \
	\leq C\delta^{-2} \beta \dashint_{B_\delta(x_0)} \chi_{\{u<2\beta\}} \dx
	\leq C\delta^{-2} \beta^{-\frac{n-1}{3}} \dashint_{B_\delta(x_0)}
	\Big((3\beta)^\frac{n+2}{6}-u^\frac{n+2}{6}\Big)_+^2 \dx
	\\&~~~~~~~~~ \ \ 
	\leq C\delta^{-2} \beta^{-\frac{n-1}{3}}
	\Bigg(\dashint_{B_\delta(x_0)}
	\left|u^\frac{n+2}{6}-\dashint_{B_\delta(x_0)} u^\frac{n+2}{6}(y) ~dy
	\right|^2\dx
	\\&~~~~~~~~~~~~~~~~~~~~~~~~
	\qquad \quad +\Bigg((3\beta)^\frac{n+2}{6}-\dashint_{B_\delta(x_0)} u^\frac{n+2}{6}(y) ~dy
	\Bigg)_+^2
	\Bigg) ~.
	\end{align*}
	Choosing $\beta$ as in \eqref{A3ChoiceBeta}, we obtain by the Poincar\'{e} inequality and the Sobolev embedding
	\begin{align*}
	&\Big|D^2 (\rho_\delta\ast (u-f_\beta(u)))\Big|(x_0)
	\\
	&~~~\leq C \beta^{-\frac{n-1}{3}} \dashint_{B_\delta(x_0)}
	\big|\nabla u^\frac{n+2}{6}\big|^2\dx
	\\&\qquad
	+C\delta^{-2} \beta^{-\frac{n-1}{3}}
	\Bigg(C\nu^\frac{n+2}{6} \dashint_{B_\delta(0)} u^\frac{n+2}{6} \dx
	\\ &\qquad\qquad\qquad\qquad\quad +C\nu^\frac{n+2}{6}\delta\Bigg(\dashint_{B_\delta(x_0)} |\nabla
	u^\frac{n+2}{6}|^2 \dx \Bigg)^\frac{1}{2}
	-\dashint_{B_\delta(x_0)} u^\frac{n+2}{6} \dx \Bigg)_+^2
	~.
	\end{align*}
	Choosing $\nu>0$ small enough (depending only on $n$ and $d$), we infer
	\begin{align*}
	&\Big|D^2 (\rho_\delta\ast (u-f_\beta(u)))\Big|(x_0)
	\leq C \beta^{-\frac{n-1}{3}}
	\dashint_{B_\delta(0)} \Big|\nabla u^\frac{n+2}{6}\Big|^2 \dx
	~.
	\end{align*}
	We next estimate $I_{22}$. To this aim, we may rewrite
	\begin{align*}
	&D^2 (\rho_\delta \ast f_\beta(u))
	=\rho_\delta \ast \Big(f_\beta'(u)D^2 u+f_\beta''(u)\nabla u \otimes \nabla u\Big),
	\end{align*}
	which implies
	\begin{align*}
	&\Big|D^2 (\rho_\delta \ast f_\beta(u))\Big|
	\\&~~~~~~~~
	\leq \rho_\delta \ast
	\Big(\beta^{-\frac{n-1}{3}}u^\frac{n-1}{3}|D^2 u|
	+C\beta^{-\frac{n-1}{3}}u^\frac{n-4}{3}|\nabla u|^2
	\Big)
	~.
	\end{align*}
	For a.e. $t \in (0,T)$, we have $\nabla u^{\frac{n+2}{6}} \in L^6(\R^d\times [0,T])$ and $u^{\frac{n}{2}} \nabla \Delta u \in L^2(\R^d\times [0,T])$. Using these facts and the computation
	\begin{align*}
	&\sum_{i,j} \int u^{n-1} |\partial_i \partial_j u|^3 \dx
	\\&
	= -2\sum_{i,j} \int u^{n-1} |\partial_i \partial_j u| \partial_j u \, \partial_i^2 \partial_j u \dx - (n-1) \sum_{i,j} \int u^{n-2} |\partial_i \partial_j u| \partial_i \partial_j u  \partial_j u \, \partial_i u \dx
	\end{align*}
	as well as H\"older's inequality,
	we can show that $u^{\frac{n-1}{3}}D^2u \in L^3(\R^d\times [0,T])$. Then we can establish convergence of the second integral on the right-hand side of \eqref{A3FirstEquation} arguing as we have done for the first integral.

The convergence of the other integrals on the right-hand side of \eqref{A3FirstEquation} in the limit $\delta\rightarrow 0$ may be shown analogously, thereby establishing Lemma~\ref{WeightedEnergy}.
\end{proof}

\vspace{3mm}
\section*{Acknowledgments}
N.\ De~Nitti acknowledges the kind hospitality of IST Austria within the framework of the \emph{ISTernship Summer Program 2018}, during which most of the present paper was written. N.\ De~Nitti has received funding by The Austrian Agency for International Cooperation in Education \& Research (OeAD-GmbH) via its financial support of the \emph{ISTernship Summer Program 2018}. N. De~Nitti would also like to thank Giuseppe Coclite, Giuseppe Devillanova, Giuseppe Florio, Sebastian Hensel, and Francesco Maddalena for several helpful conversations on topics related to this work.

\vspace{3mm}
\bibliographystyle{abbrv}
\bibliography{thinfilm}

\end{document}